\documentclass[AMA,STIX1COL]{WileyNJD-v2}

\def\etal.{et\penalty50\ al.}

\usepackage{dsfont, url,bm}

\articletype{Research Article}%

\received{dd Month yyyy}
\revised{dd Month yyyy}
\accepted{dd Month yyyy}

\begin{document}

\title{Optimal Power Consumption for Demand Response of Thermostatically Controlled Loads}

\author[1]{Abhishek Halder*}

\author[2]{Xinbo Geng}

\author[3]{Fernando A.C.C. Fontes}

\author[2]{P.R. Kumar}

\author[2]{Le Xie}

\authormark{HALDER \textsc{et al}}

\address[1]{\orgdiv{Department of Applied Mathematics and Statistics}, \orgname{University of California, Santa Cruz}, \orgaddress{\state{CA}, \country{USA}}}

\address[2]{\orgdiv{Department of Electrical and Computer Engineering}, \orgname{Texas A\&M University}, \orgaddress{\state{TX}, \country{USA}}}

\address[3]{\orgdiv{ISR-Porto and Faculdade de Engenharia}, \orgname{Universidade do Porto}, \orgaddress{\state{Porto}, \country{Portugal}}}

\corres{*Abhishek Halder, Department of Applied Mathematics and Statistics, University of California, Santa Cruz, CA 95064. \email{ahalder@ucsc.edu}}

%\presentaddress{This is sample for present address text this is sample for present address text}

\abstract[Summary]{We consider the problem of determining
the optimal aggregate power consumption of a population of thermostatically controlled loads such as air conditioners. This is motivated by the need to synthesize the demand response for a load serving entity (LSE) catering a population of such customers. We show how the LSE can opportunistically design the aggregate reference consumption to minimize its energy procurement cost, given day-ahead price, load and ambient temperature forecasts, while respecting each individual load's comfort range constraints. The resulting synthesis problem is intractable when posed as a direct optimization problem after Euler discretization of the dynamics, since it results in a mixed integer linear programming problem with number of variables typically of the order of millions. In contrast, in this paper we show that the problem is amenable to continuous-time optimal control techniques. Numerical simulations elucidate how the LSE can use the optimal aggregate power consumption trajectory thus computed, for the purpose of demand response.}


% \jnlcitation{\cname{%
% \author{A. Halder},
% \author{X. Geng},
% \author{F. A.C.C. Fontes},
% \author{P.R. Kumar}, and
% \author{L. Xie}} (\cyear{2018}),
% \ctitle{Optimal Power Consumption for Demand Response of Thermostatically Controlled Loads}, \cjournal{Optimal Control Applications and Methods}, \cvol{2018}.}

\maketitle

%\footnotetext{\textbf{Abbreviations:} ANA, anti-nuclear antibodies; APC, antigen-presenting cells; IRF, interferon regulatory factor}

\section{Introduction}\label{sec1}

Motivated by the goal of sustainable electricity generation, renewables
such as solar and wind are of increasing interest as electric energy resources. Concomitantly, the inherent time variability of such renewable generation is shifting modern power system operation from the traditional ``supply follows demand'' paradigm to the one where ``demand adapts to supply''. This new operational paradigm, called ``demand response'' \cite{DOE2006Report,callaway2011achieving}, can leverage demand side flexibility to offset variability in generation. Of particular interest in this context are thermostatically controlled loads (TCLs) such as air conditioners.
In this paper we examine how an ``aggregator'', also known as a ``load serving entity'' (LSE), can employ
a population of its
customers' TCLs to shape the aggregate power consumption, while adhering to each load's comfort constraints.
We consider an LSE buying energy from the day-ahead market for a population of TCLs.
We address the question of designing the optimal aggregate power consumption trajectory for this population, given a forecast of day-ahead price trajectory.

\subsection*{Related Work}
Modeling the dynamics of a population of TCLs has been investigated in several papers \cite{ChongDebs1979, MalhameChong1985, Callaway2009Energy, BashashFathy2011ACC, KunduSinitsynBackhausHiskens2011PSCC, MathieuKochCallaway2013TPS, Zhang2013TPS, TotuWisniewski2014}, with the aim of deriving control-oriented models that can accurately predict the aggregate power trajectory.
Once such a model is obtained, the predominant focus in these and other papers \cite{GhaffariMouraKrstic2014DSMC,Grammatico2015ECC,Meyn2015TAC} is to design a model-based setpoint controller to enable the TCL population track a given reference aggregate power trajectory in real-time, thereby compensating for the possible mismatch between the real-time and forecasted ambient temperatures. %Although the availability of a reference power trajectory is tacitly assumed in the literature, in practice the LSE needs to determine this reference to minimize the energy procurement cost while guaranteeing that the reference trajectory can be tracked by the aggregate dynamics in such a manner that  individual comfort range constraints are met.
In contrast to these papers, where the availability of a reference power trajectory is assumed for real-time control design, we focus on the case where the LSE determines this reference to minimize its energy procurement cost while guaranteeing that the reference trajectory can indeed be tracked by the aggregate dynamics without violating individual comfort range constraints.

In Paccagnan \etal. \cite{Paccagnan2015Range}, the range of feasible reference power trajectories was studied. %To the best of the authors' knowledge, the present paper is first to focus on determining the optimal reference power consumption while respecting end users' comfort zones.
Reulens \etal. \cite{Ruelens2014PSCC} adopted a model-free approach to schedule a cluster of electric water heaters using batch reinforcement learning. Also relevant to our work is the paper by Mathieu \etal. \cite{mathieu2015arbitraging} where minimizing TCL energy consumption cost subject to end users' comfort zones was considered (see equation (5) in Section III of that reference). The perspective and results of our paper differ significantly from the aforesaid formulation in that we allow a target total energy budget constraint for the LSE, which in turn prohibits transcribing the overall (discrete version of the) optimization problem into a set of decoupled mixed integer linear programs (MILPs), as was the case in Mathieu \etal.\cite{mathieu2015arbitraging} In fact, it is precisely the dynamic coupling that makes the non-convex optimal control problem difficult to solve by a direct ``discretize-then-optimize'' approach, as we explain further in Section 3.3.

\subsection*{Contributions of This Paper}
For an LSE managing a finite population of TCLs, we formulate and solve the optimal aggregate power consumption design as a finite horizon deterministic optimal control problem. The contribution of the present paper beyond our previous work \cite{AbhishekSGC2015, AbhishekTPS2017} is that herein, we analytically solve the continuous-time optimal control problem (Section IV), thereby revealing qualitative insights on how the LSE can use the knowledge of day-ahead price forecast, load forecast, and ambient temperature forecast, for the purpose of energy procurement at least cost. Furthermore, when there is additional constraint on minimum thermostatic switching period, we provide an algorithm (Section V and VI) to recover the optimal binary controls from the corresponding convexified optimal control solutions.

In the presence of state inequality constraints arising from comfort range contracts between the LSE and individual TCLs, the optimal controls are shown to depend on both the shape of the day-ahead price trajectory, and on the minimum switching period (also known as ``lockout constraint'' \cite{Zhang2013TPS}) at the upper and lower comfort boundaries allowable by the thermostats. Specifically, the application of Pontryagin's maximum principle (PMP) reveals that the optimal policy is a function of certain ``threshold price'' to be computed from the day-ahead price forecast. The resulting optimal indoor temperature trajectories are described in terms of the so-called ``two-sided Skorokhod maps'' \cite{Kruk2007,Kruk2008} parameterized by individual TCL's upper and lower comfort boundaries.

This paper is organized as follows. In Section II, we describe the mathematical models. In Section III, we formulate the design of power consumption as  an optimal control problem. Sections IV and V present the solution of the power consumption design problem. In Section VI, numerical results based on the day-head price forecast data from Electric Reliability Council of Texas (ERCOT) and the ambient temperature forecast data from a weather station in Houston, Texas are reported, to illustrate how the LSE can use the optimal power consumption trajectory computed via the proposed framework, for the purpose of demand response. Section VII concludes the paper.

\subsection*{Notation}
We use the symbols $\mathds{1}_{X}$ and $|X|$ to respectively denote the indicator function, and the Lebesgue measure of set $X$. The set of integers, reals and positive reals are denoted by $\mathbb{Z}, \mathbb{R}$ and $\mathbb{R}^{+}$, respectively. We recall that c\`adl\`ag functions are defined to be everywhere right-continuous functions having left limits everywhere, and use $\mathbb{D}(Y)$ to denote the space of c\`adl\`ag functions whose range is set $Y$. For $a,b\in\mathbb{R}$, we use the notations $a\vee b := \max(a,b)$, $a\wedge b := \min(a,b)$, $\left[a\right]^{+} := 0 \vee a$, and $\lceil a \rceil := \min\{n\in\mathbb{	Z} : n \geq a\}$. The symbol $\circ$ denotes the composition operator, and spt$(\cdot)$ denotes the support of a function.

%%%%%%%%%%%%%%%%%%%%%%%%%%%%%%%%%%%%%%%%%%%%%%%%%%%%%%%%%%%%%%%%%%%%%%%%%%%%%%%%

\section{Model}

\begin{figure}[t]
  \centering
    \includegraphics[width=0.49\textwidth]{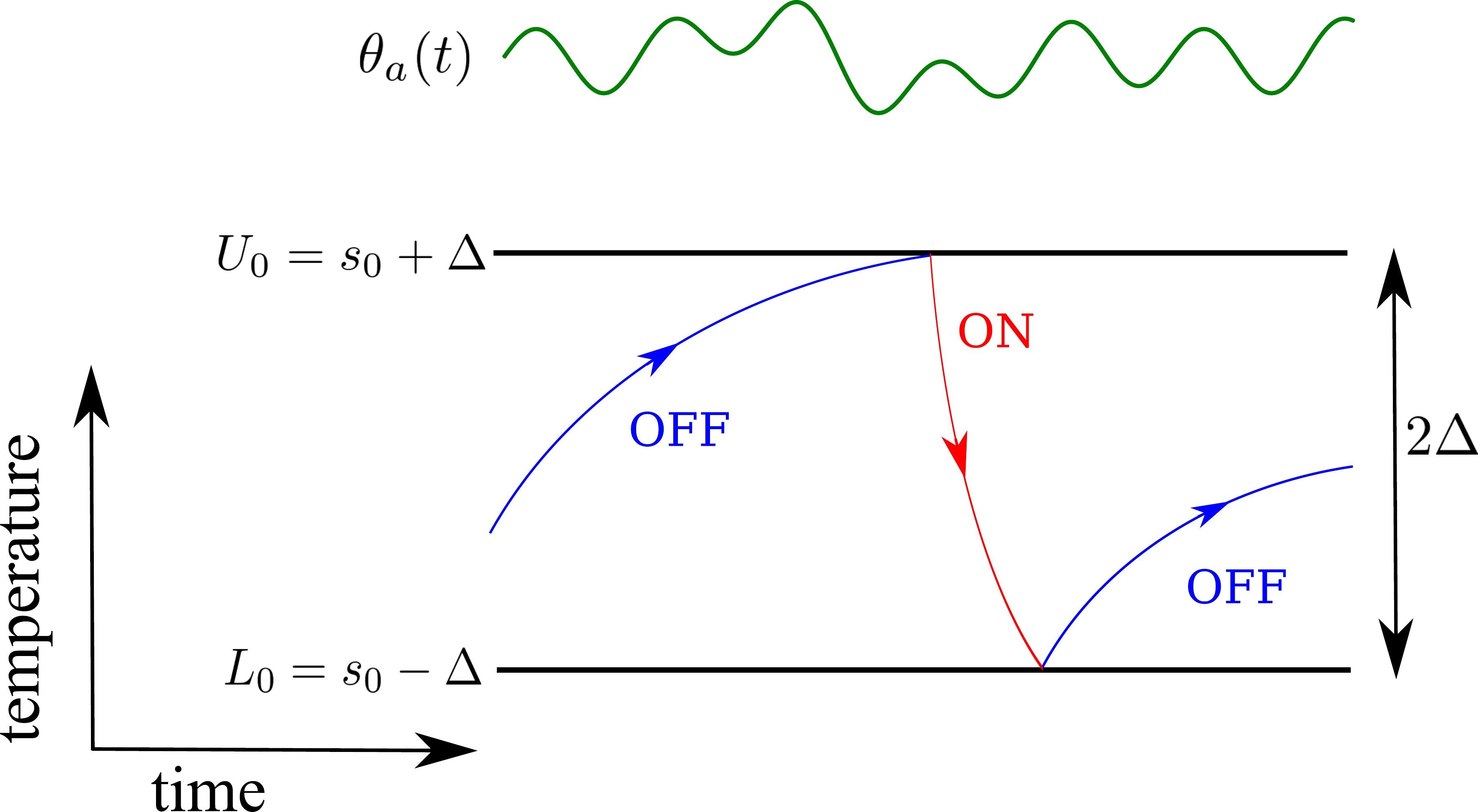}
    \caption{The dynamic behavior of a TCL with fixed setpoint $s_{0}$, is illustrated for a time-varying ambient temperature $\theta_{a}(t)$. The comfort temperature interval for the TCL is $[L_{0},U_{0}]$ with range $U_{0}- L_{0} = 2\Delta$. The indoor temperature trajectory $\theta(t)\in[L_{0},U_{0}]$ consists of alternating OFF (\emph{blue, up-going}) and ON (\emph{red, down-going}) segments, with the boundaries $L_{0}$ and $U_{0}$ acting as reflecting barriers for the ON and OFF segments, respectively.}
    \label{TCLDynamics}
\end{figure}

\subsection{Dynamics of Individual Thermostatically Controlled Load}
The dynamic behavior of an individual TCL is shown in Fig. \ref{TCLDynamics}\!\!\!. At time $t$, let us denote the indoor temperature by $\theta(t)$, and the ambient temperature by $\theta_{a}(t)$. At $t=0$, an occupant privately sets a temperature $s_{0}:=s(0)$, called setpoint, close to which the indoor temperature $\theta(t)$ must lie at all times. If the occupant is willing to tolerate at most $\pm \Delta$ temperature deviation from $s_{0}$, then we define its ``temperature comfort range" as $[L_{0},U_{0}] := [s_{0}-\Delta,s_{0}+\Delta]$. If the setpoint $s$ does not change with time, then $s(t) \equiv s_{0}$, and consequently, the comfort boundaries $L_{0}$ and $U_{0}$ remain fixed over time. For specificity we consider the problem of controlling air-conditioning rather than heating, though the theory
developed in the sequel applies to both.

The rate of change of $\theta(t)$ is governed by Newton's law of heating/cooling given by the ordinary differential equation (ODE)
\begin{eqnarray}
\dot{\theta}(t) = -\alpha\left(\theta(t) - \theta_{a}(t)\right) - \beta P \sigma(t),
\label{individualTCLdynamics}	
\end{eqnarray}
where $\sigma(t)$ is the ON/OFF mode indicator variable of the air-conditioner, given by
\begin{eqnarray}
\sigma(t) := \begin{cases}
	1 \qquad\qquad\text{if}\quad\theta(t) = U_{0},\\
	0 \qquad\qquad\text{if}\quad\theta(t) = L_{0},\\
	\sigma\left(t^{-}\right)\quad\quad\text{otherwise}.
 \end{cases}
\label{sigmaDynamics}	
\end{eqnarray}
In other words, $\sigma(t)=1 (0)$ indicates that the TCL is in ON (OFF) mode. In (\ref{individualTCLdynamics}), the parameters $\alpha, \beta, P >0$ respectively denote the heating time constant, thermal conductivity, and amount of thermal power drawn by the TCL in ON mode. A parameter $\eta>0$ called load efficiency, relates the thermal power drawn $P$, with the electrical power drawn $P_{e}$, via the formula $P_{e} = \frac{P}{\eta}$. The state of a TCL at time $t$ is the tuple $\{s(t), \theta(t), \sigma(t)\} \in \mathbb{R}^{2} \times \{0,1\}$.

As shown in Fig. \ref{TCLDynamics}\!\!\!, starting from an initial condition $(s_{0},\theta_{0},\sigma_{0}=0)$, the indoor temperature $\theta(t)$ rises exponentially until it hits the upper boundary $U_{0}$, at which time an OFF$\rightarrow$ON mode transition occurs, and subsequently $\theta(t)$ decreases exponentially until it hits the lower boundary $L_{0}$, at which time an ON$\rightarrow$OFF transition takes place, and so on. Thus, the dynamics of a TCL is hysteretic in the sense that if $\theta_{0} \in [L_{0},U_{0}]$, then $\theta(t) \in [L_{0},U_{0}]$ for all $t > 0$. While the qualitative behavior shown in Fig. \ref{TCLDynamics}\!\!\! is true for any $\theta_{a}(t) > U_{0}$, temporal variations of $\theta_{a}(t)$ engender time-varying heating/cooling rates for $\theta(t)$.

\begin{figure}[t]
  \centering
    \includegraphics[width=0.69\textwidth]{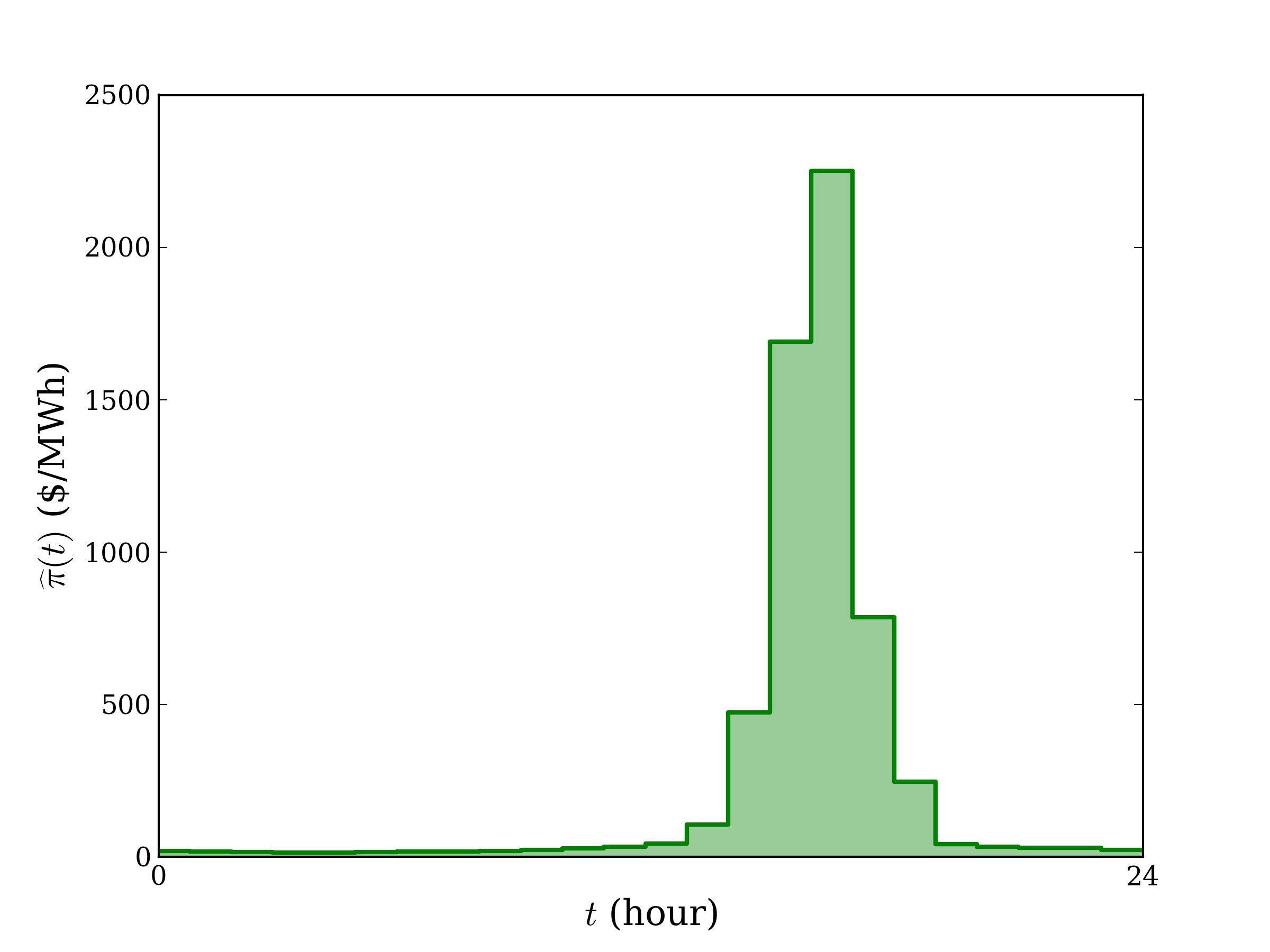}
    \caption{A typical day-ahead price forecast trajectory $\widehat{\pi}(t)$ is increasing till the late afternoon, and decreasing thereafter. The shown trajectory is for Houston on August 10, 2015, available from the day-ahead energy market of Electric Reliability Council of Texas (ERCOT)\cite{DApriceERCOT}.}
    \label{DApriceforecast}
\end{figure}

\subsection{Day Ahead Price Forecasts}
We suppose that the LSE is exposed to a price forecast $\widehat{\pi}(t)$ and ambient temperature forecast $\widehat{\theta}_{a}(t)$, over a time horizon $[0,T]$. For example, if $T = 24$ hours, and the forecast is made on the
previous day, then $\widehat{\pi}(t)$ is the forecasted price from the day-ahead energy market. We also allow
 the LSE to have a target for the total energy $E$ to be consumed over $[0,T]$ by the population of $N$ TCL customers managed by that LSE. The choice of $E$ may be restricted by the parameters of the TCL population, an issue we address in Section \ref{FeasibilitySubsection}. To minimize its energy procurement cost, the LSE would like to schedule purchase of energy when $\widehat{\pi}(t)$ is low, and defer purchase when $\widehat{\pi}(t)$ is high, while satisfying the total energy budget ($E$), as well as maintaining the comfort constraints described below
 specified by
 each of the TCLs.

A typical day-ahead price forecast trajectory $\widehat{\pi}(t)$ is shown in Fig. \ref{DApriceforecast}\!\!\!. Since the aggregate power consumption around late afternoon is expected to be higher than at other times of the day, the day-ahead price is typically forecasted to be increasing till late afternoon and decreasing thereafter.

\subsection{Comfort Range Contracts}
Each of the $N$ TCLs managed by an LSE, may have different comfort ranges $[L_{i0},U_{i0}]$ with different tolerances $\Delta_{i}$, $i=1,\hdots,N$. Let $\theta_{i}(t)$ be the indoor temperature of the $i$\textsuperscript{th} home at time $t$. The LSE is obligated to maintain the indoor temperatures
of its customer TCLs within their specified comfort ranges.
That is, for each $i=1,\hdots,N$, the LSE must ensure that $\theta_{i}(t) \in [L_{i0}, U_{i0}]$ for all $t\geq0$. Such an agreement constitutes a \emph{contract} between the LSE and an individual TCL.

An important part of this agreement is the \emph{flexibility} of a load, captured by its range $2 \Delta$. The LSE can utilize this flexibility to
optimally time its purchase of power. The LSE's business model essentially consists of sharing part of the realized
savings with the customers in terms of serving their needs for energy at low cost. Naturally, a customer with a greater flexibility
$2 \Delta$ is more valuable to the LSE and such customers can obtain better contracts from the LSE.
%An important question that arises is how to quantify the value of a customer as a function of its flexibility $2 \Delta$.

\subsection{Assumptions}\label{assumptions}
For rest of this paper, we make the following assumptions.
\begin{itemize}
\item The ambient temperature forecast $\widehat{\theta}_{a}(t)$, and price forecast $\widehat{\pi}(t)  >0$, are continuous functions of time $t$.

\item All TCLs are cooling, i.e., for all $t \in[0,T]$, we have $\widehat{\theta}_{a}(t)> \displaystyle\max_{i=1,\hdots,N}U_{i0}$.

\item Each TCL in the population, when ON, draws the same thermal power $P$. Further, each TCL is assumed to have same load efficiency $\eta$.

\item Without loss of generality, the initial indoor temperatures $\theta_{i0} := \theta_{i}(0) \in [L_{i0},U_{i0}]$, for all $i=1,\hdots,N$.
	
\end{itemize}

%%%%%%%%%%%%%%%%%%%%%%%%%%%%%%%%%%%%%%%%%%%%%%%%%%%%%%%%%%%%%%%%%%%%%%%%%%%%%%%%

%%%%%%%%%%%%%%%%%%%%%%%%%%%%%%%%%%%%%%%%%%%%%%%%%%%%%%%%%%%%%%%%%%%%%%%%%%%%%%%%

\section{Problem Formulation}
Consider an LSE managing $N$ TCLs with thermal coefficients $\{\alpha_{i},\beta_{i}\}_{i=1}^{N}$, initial conditions $\big\{\left(s_{i0}, \theta_{i0}, \sigma_{i0}\right)\big\}_{i=1}^{N}$, and comfort tolerances $\{\Delta_{i}\}_{i=1}^{N}$.
We now formulate the optimal power consumption design problem. We suppose that the LSE has available
estimates \cite{Moura2013Observer,Moura2014Parameter} of the parameters and initial conditions at the beginning of the time horizon.

\subsection{The Load Serving Entity's Objective}
\label{LSEobjectiveSubsection}
Denoting by $n_{\text{ON}}(t)$ the number of ON TCLs at time $t$, the aggregate electrical power drawn by the TCL population at time $t$ is
\begin{eqnarray*}
P_{e}n_{\text{ON}}(t) = \frac{P}{\eta}\sum_{i=1}^{N}u_{i}(t).	
\end{eqnarray*}
We take the switching trajectories $\sigma_{i}(t)$ as decision variables $u_{i}(t) \in \{0,1\}$, and introduce an extended state vector
\begin{eqnarray*}
	\bm{x}(t) := \left( \theta_{1}(t), \:\hdots,\: \theta_{N}(t),\: t,\: \int_{0}^{t} \sum_{i=1}^{N}u_{i}(\varsigma)\:\mathrm{d}\varsigma \right)^{\top}
\end{eqnarray*}
of size $(N+2)\times 1$. At time $t$, the components of $\bm{x}(t)$ are $x_{i}(t):=\theta_{i}(t)$ for $i=1,\hdots,N$, $x_{N+1}(t):=t$, and $x_{N+2}(t):=\int_{0}^{t} \sum_{i=1}^{N}u_{i}(\varsigma)\:{\mathrm{d}}\varsigma$. Then, to minimize the procurement cost for total energy consumption over $[0,T]$, the LSE needs to
\begin{equation}
\underset{u_{1}(t),\hdots,u_{N}(t)\in\{0,1\}^{N}} {\text{minimize}}	\quad \displaystyle\frac{P}{\eta} \displaystyle\int_{0}^{T} \widehat{\pi}(t) \displaystyle\sum_{i=1}^{N}u_{i}(t) \: \mathrm{d}t, \label{PlanningCostFcn}
\end{equation}
subject to the constraints:\\
\noindent \textbf{C1.} (Indoor temperature dynamics)
\begin{subequations}
\begin{align}
&\dot{x}_{i}(t) = -\alpha_{i} (x_{i}(t) - \widehat{\theta}_{a}(t)) - \beta_{i} P u_{i}(t), \quad x_{i}(0) = \theta_{i0}, \qquad i=1,\hdots,N,\label{DynamicsConstraintInitial}\\
&\dot{x}_{N+1}(t) = 1, \qquad\qquad\qquad\qquad\qquad\;\: x_{N+1}(0) = 0, \quad\, x_{N+1}(T) = T,
\label{DynamicsConstraintFinal}
\end{align}
\label{IndoorTempDynConstraint}
\end{subequations}
\noindent \textbf{C2.} (Energy/isoperimetric constraint)
\begin{flalign}
&\dot{x}_{N+2}(t) = \sum_{i=1}^{N}u_{i}(t), \qquad x_{N+2}(0) = 0, \qquad x_{N+2}(T) = \frac{\eta E}{NP},
	\label{EnergyBudgetConstraint}
\end{flalign}
\noindent \textbf{C3.} (Contractual comfort/state inequality constraint)
\begin{eqnarray}
	L_{i0} \leq x_{i}(t) \leq U_{i0},
	\label{ComfortRangeConstraint}
\end{eqnarray}
where $[L_{i0},U_{i0}] := [s_{i0} - \Delta_{i}, s_{i0} + \Delta_{i}]$ for $i = 1,\hdots,N$.

Notice that the constraints (\ref{IndoorTempDynConstraint}) and (\ref{ComfortRangeConstraint}) are decoupled, while the cost function (\ref{PlanningCostFcn}) and constraint (\ref{EnergyBudgetConstraint}) are coupled. Denoting the solution of the open loop deterministic optimal control problem (\ref{PlanningCostFcn})--(\ref{ComfortRangeConstraint}) by $\{u_{i}^{*}(t)\}_{i=1}^{N}$, the optimal power consumption trajectory is given by $P_{\text{total}}^{\text{ref}}(t) = \frac{P}{\eta}\sum_{i=1}^{N}u_{i}^{*}(t)$.

\begin{remark}
	The optimal control problem (\ref{PlanningCostFcn})--(\ref{ComfortRangeConstraint}) is non-autonomous since there are explicit dependences on time in the cost function (via $\widehat{\pi}(t)$) and in the dynamics (via $\widehat{\theta}_{a}(t)$). This motivates the inclusion of time $t$ as a component of the extended state vector $\bm{x}(t)$.
\end{remark}

\subsection{Feasibility}
\label{FeasibilitySubsection}
Let $\tau := \frac{\eta E}{P}$, and notice that the constraint (\ref{EnergyBudgetConstraint}) imposes a necessary condition for feasibility,
\begin{eqnarray}
0 \leq \overline{\tau} := \displaystyle\frac{\tau}{NT} = \displaystyle\frac{\eta E}{NPT} = \displaystyle\frac{x_{N+2}(T)}{T} \leq 1.
\label{Feasibility}	
\end{eqnarray}
Given an ambient temperature forecast $\widehat{\theta}_{a}(t)$ and parameters of the TCL population, further restriction of $\overline{\tau}$ is needed to include the possibility of zero dynamics on (meaning the temperature trajectory chatters along) the boundaries $L_{i0}$ and $U_{i0}$. Such a restriction is of the form
\begin{eqnarray}
0 \leq \overline{\tau}_{\ell} \leq \overline{\tau} \leq \overline{\tau}_{u} \leq 1,
\label{FeasibilitywZeroDynamics}	
\end{eqnarray}
where $\overline{\tau}_{\ell} := \frac{\tau_{\ell}}{NT} = \frac{\eta E_{\ell}}{NPT}$, and likewise for $\overline{\tau}_{u} = \frac{\eta E_{u}}{NPT}$. Here $E_{\ell}$ (resp. $E_{u}$) is the aggregate energy consumed if the entire population were to be maintained at their private upper (lower) setpoint boundaries, thus resulting in the lowest (highest) total energy consumption while respecting (\ref{ComfortRangeConstraint}). In other words, $E_{\ell} = \frac{P}{\eta}\sum_{i=1}^{N}\int_{0}^{T}u_{i}(t)\:\mathrm{d}t$, where the zero dynamics controls are $u_{i}(t) = \frac{\alpha_{i}}{\beta_{i}P}\left(\widehat{\theta}_{a}(t) - U_{i0}\right)$, and hence
\begin{eqnarray}
\overline{\tau}_{\ell} = \displaystyle\frac{1}{NP}\left(\displaystyle\sum_{i=1}^{N}\frac{\alpha_{i}}{\beta_{i}}\left(\langle\widehat{\theta}_{a}\rangle - U_{i0}\right)\right),
\label{taulowernormalized}	
\end{eqnarray}
where $\langle\widehat{\theta}_{a}\rangle := \frac{1}{T}\int_{0}^{T}\widehat{\theta}_{a}(t)\:\mathrm{d}t$. Similar calculation yields
\begin{eqnarray}
\overline{\tau}_{u} = \displaystyle\frac{1}{NP}\left(\displaystyle\sum_{i=1}^{N}\frac{\alpha_{i}}{\beta_{i}}\left(\langle\widehat{\theta}_{a}\rangle - L_{i0}\right)\right).
\label{tauuppernormalized}
\end{eqnarray}
Thus, (\ref{FeasibilitywZeroDynamics}) characterizes the necessary and sufficient conditions for feasibility of the optimal control problem (\ref{PlanningCostFcn})--(\ref{ComfortRangeConstraint}).

\begin{remark}
Notice that constraint \textbf{C2} expresses a total energy budget $\frac{P}{\eta}\int_{0}^{T}\sum_{i=1}^{N}u_{i}(t)\mathrm{d}t = E$. In the absence of \textbf{C2}, the optimal controls $\{u_{i}^{*}(t)\}_{i=1}^{N}$ that minimize (\ref{PlanningCostFcn}) subject to \textbf{C1} and \textbf{C3}, satisfy $\frac{P}{\eta}\int_{0}^{T}\sum_{i=1}^{N}u_{i}^{*}(t)\mathrm{d}t = E_{\ell}$, where from (\ref{taulowernormalized}), we have
\begin{eqnarray}
E_{\ell} = \frac{T}{\eta}\left(\displaystyle\sum_{i=1}^{N}\frac{\alpha_{i}}{\beta_{i}}\left(\langle\widehat{\theta}_{a}\rangle - U_{i0}\right)\right).
\label{EnergyLower}	
\end{eqnarray}
\end{remark}

\begin{remark}
Notice also that condition (\ref{FeasibilitywZeroDynamics}) is equivalent to the energy inequality $E_{\min} \leq E_{\ell} \leq E \leq E_{u} \leq E_{\max}$, where $E_{\min} := 0$, $E_{\max} := \frac{NPT}{\eta}$, $E_{\ell}$ is given by (\ref{EnergyLower}), and from (\ref{tauuppernormalized}) we have $E_{u} = \frac{1}{NP}(\sum_{i=1}^{N}\frac{\alpha_{i}}{\beta_{i}}(\langle\widehat{\theta}_{a}\rangle - L_{i0}))$. As a consequence, the feasible energy budget $E$ must belong to an interval $[E_{\ell},E_{u}]$ with length $E_{u} - E_{\ell} = \frac{2T}{\eta}\sum_{i=1}^{N}\frac{\alpha_{i}}{\beta_{i}}\Delta_{i}$.
\end{remark}

\subsection{Difficulty in Direct Numerical Simulation}
\label{DifficultyDNS}
A direct numerical approach converts the optimal control problem (\ref{PlanningCostFcn})--(\ref{ComfortRangeConstraint}) to an optimization problem via time discretization. Such a ``discretize-then-optimize" strategy leads to a mixed integer linear program (MILP), since $\theta_{i}(t) \in \mathbb{R}$ and $u_{i}(t) \in \{0,1\}$, for $i=1,\hdots,N$. Typically, the day-ahead price forecast $\widehat{\pi}(t)$ is available as a function that is piecewise constant for each hour, and so taking the Euler discretization for dynamics (\ref{DynamicsConstraintInitial}) with 1 minute time resolution results in an MILP with $24\times 60\times 2 \times N$ variables. In our experience, solving the MILP even for $N=2$ homes for the day-ahead price, is computationally expensive (with CPU runtime over 24 hours) in Gurobi \cite{Gurobi2015}. On the other hand, a linear program (LP) relaxation of the MILP, resulting from the control convexification $u_{i}\in \{0,1\} \mapsto \widetilde{u}_{i} \in [0,1]$, has much faster runtime and was reported in our earlier work \cite{AbhishekSGC2015}. Furthermore, the optimal solution of the LP relaxation has the physical meaning of average ON duration over a discretization interval (see Fig. \ref{MILPvsLP}\!\!\!).

\begin{figure}[t]
  \centering
    \includegraphics[width=0.65\textwidth]{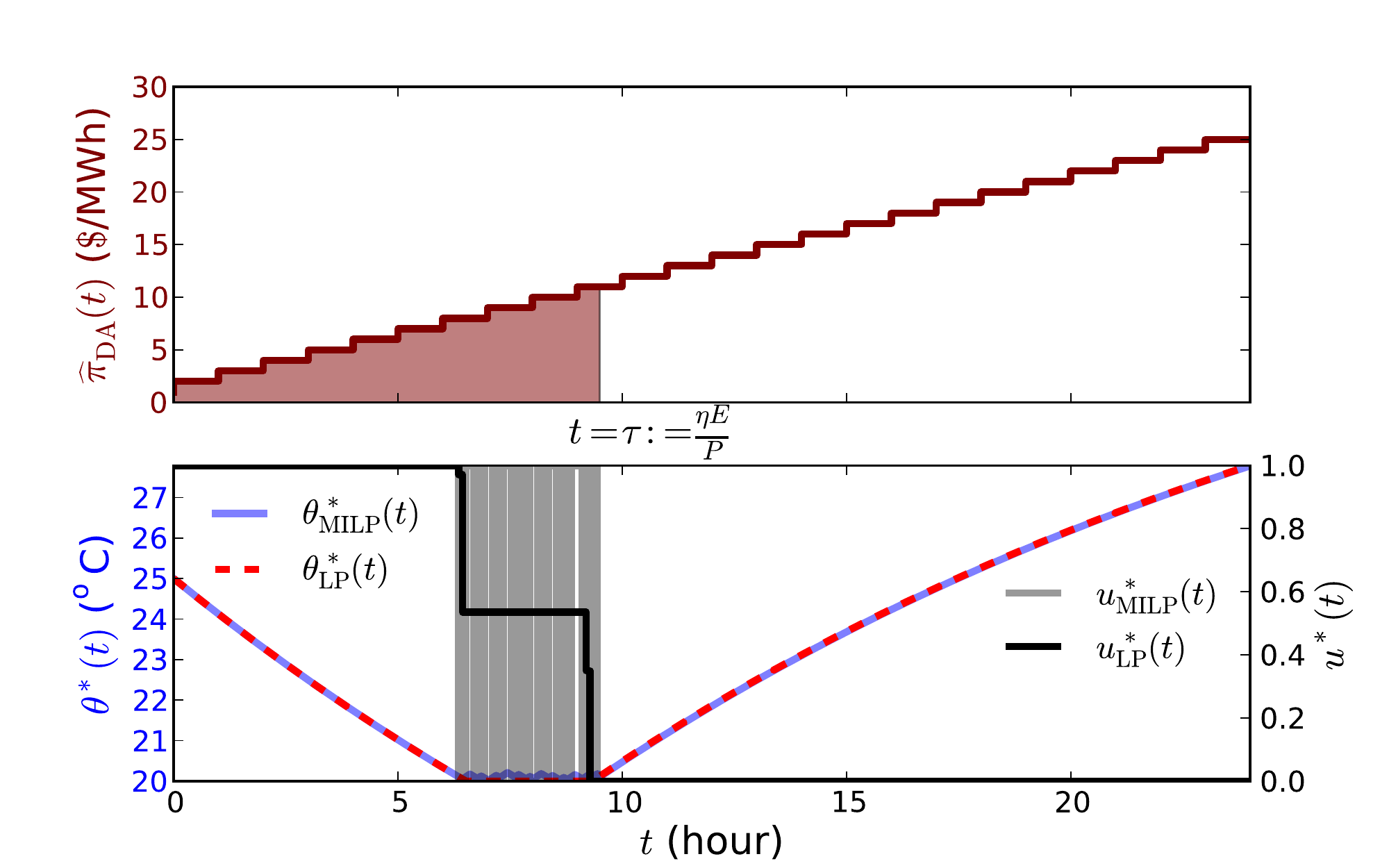}
    \caption{Top: A monotone increasing day-ahead price forecast $\widehat{\pi}_{\text{DA}}(t)$ that is piecewise constant for each hour, is shown with $\tau=\frac{\eta E}{P}$. Bottom: The corresponding solution of the optimal power consumption problem (\ref{PlanningCostFcn})--(\ref{ComfortRangeConstraint}) via discretized MILP (from Gurobi) and LP (from MATLAB \texttt{linprog}) for $N=1$ home with $[L_{0},U_{0}] = [20^{\circ}\mathrm{C},30^{\circ}\mathrm{C}]$, and constant ambient forecast $\widehat{\theta}_{a}(t) = 32^{\circ}\mathrm{C}$. The parameter values used in this simulation are from Table 1, p. 1392 in Callaway \cite{Callaway2009Energy}.}
    \label{MILPvsLP}
    %\vspace*{-0.2in}
\end{figure}

However, the MILP equality constraint does not satisfy total unimodularity (see for example, Ch. 5 in Schrijver \cite{schrijver2002combinatorial}); consequently the optimal solution of the LP relaxation is not the optimal solution of the MILP. In the following Sections IV and V, we solve the continuous time optimal control problem (\ref{PlanningCostFcn})--(\ref{ComfortRangeConstraint}) using Pontryagin's maximum principle (PMP) \cite{PontryaginBook}, thereby obtaining qualitative insights into the optimal solution, which are otherwise difficult to gauge from the direct numerical solution of the discretized LP relaxation.

%%%%%%%%%%%%%%%%%%%%%%%%%%%%%%%%%%%%%%%%%%%%%%%%%%%%%%%%%%%%%%%%%%%%%%%%%%%%%%%%

\section{Solution of the Optimal Control Problem}
For the optimal control problem (\ref{PlanningCostFcn})--(\ref{ComfortRangeConstraint}), if we remove constraints \textbf{C2} and \textbf{C3}, then the optimal control is trivial: $u_{i}^{*}(t) = 0$ for all $i=1,\hdots,N$, for all $t \in [0,T]$. In Section \ref{C1C2subsection}, we first discuss the non-trivial case of solving (\ref{PlanningCostFcn}) subject to \textbf{C1} and \textbf{C2}, i.e., in the absence of the inequality constraints \textbf{C3}. This is followed up with the solution for general case in Section \ref{ZeroAmpSoln}, with all constraints \textbf{C1}--\textbf{C3} active.

\subsection{Solution with Constraint C3 Inactive}
\label{C1C2subsection}

The following Theorem summarizes our results for this case. In particular, it reveals key structural properties of the optimal solution, viz. (i) the optimal controls are synchronizing across the TCL population, (ii) the optimal policy is of threshold type, i.e., there is a unique threshold price $\pi^{*}$ to be determined from the given day-ahead price trajectory $\widehat{\pi}(t)$ such that the optimal control is ON (resp. OFF) whenever $\widehat{\pi}(t)$ falls below (resp. exceeds) that threshold, (iii) the computation of $\pi^{*}$ amounts to performing monotone rearrangement (p. 276, Ch. 10 in Hardy \etal. \cite{HardyLittlewoodPolyaBook}) of the trajectory $\widehat{\pi}(t)$ and allocating the requisite ON time $\frac{\tau}{N}$ in the interval $[0,T]$ such a way that corresponds to the least price segment (this will be further elaborated in Remark \ref{MonotoneRearrangmentRemark} following the proof).

\begin{theorem}
\label{12theorem}
Consider problem (\ref{PlanningCostFcn}) with constraints \textbf{C1}--\textbf{C2}. Then
\begin{enumerate}
\item[(i)] the optimal controls are synchronizing, i.e., $u_{1}^{*}(t) = u_{2}^{*}(t) = \hdots = u_{N}^{*}(t)$ at each $t\in[0,T]$;

\item[(ii)] there is a unique threshold price $\pi^{*}$, such that $u_{i}^{*}(t)=1(0)$ for all $i=1,\hdots,N$, iff $\widehat{\pi}(t) < (\geq) \pi^{*}$;

\item[(iii)] let $\Phi_{\widehat{\pi}}\left(\widetilde{\pi}\right) := \int_{0}^{T} \mathds{1}_{\{\widehat{\pi}(t) \leq \widetilde{\pi}\}} \: \mathrm{d}t,\; \pi^{*} := \inf \{\widetilde{\pi}\in\mathbb{R}^{+} : \Phi_{\widehat{\pi}}(\widetilde{\pi}) = \frac{\tau}{N}\}$, $S := \big\{s\in[0,T] : \widehat{\pi}(s) < \pi^{*} \big\}$, and let $\bm{\lambda}$ be the costate vector corresponding to the extended state vector $\bm{x}$. The optimal solution is
\begin{align*}
&u_{i}^{*}(t) = \begin{cases}1 &\forall t\in S\\0 & \text{otherwise},\end{cases}\\
&\lambda_{i}^{*}(t) = 0\:\forall i=1,\hdots,N,\quad\lambda_{N+1}^{*}(t) = \begin{cases}\frac{P}{\eta}N\left(\pi^{*} - \widehat{\pi}(t)\right) &\forall t\in S, \\
 	 0 &\text{otherwise},
 \end{cases}\quad \lambda_{N+2}^{*}(t) = -\frac{P}{\eta}\pi^{*},\\
&x_{i}^{*}(t) = e^{-\alpha_{i}t}\left(\theta_{i0} + \int_{0}^{t}e^{\alpha_{i}\varsigma}\left(\alpha_{i}\widehat{\theta}_{a}(\varsigma) - \beta_{i}Pu_{i}^{*}(\varsigma)\right)\:\mathrm{d}\varsigma\right) \quad
\forall i=1,\hdots,N,\quad x_{N+1}^{*}(t) = t, \quad x_{N+2}^{*}(t) =   \begin{cases}N t &\forall t\in S, \\
 	 0 &\text{otherwise}.	
 \end{cases}
\end{align*}
\end{enumerate}	
\end{theorem}

\begin{proof}

\begin{enumerate}

	\item[(i)] The Hamiltonian
\begin{eqnarray}
H = \displaystyle\sum_{i=1}^{N}\left[u_{i}(t)\left(\displaystyle\frac{P}{\eta}\widehat{\pi}(t) - P\beta_{i}\lambda_{i}(t) + \lambda_{N+2}(t)\right) - \alpha_{i }\lambda_{i}(t)\left(x_{i}(t) - \widehat{\theta}_{a}(t)\right)\right] + \lambda_{N+1}(t)
\label{Hamiltonian12}	
\end{eqnarray}
gives the first order optimality conditions	
%\begin{eqnarray}
\begin{align}
&\dot{\lambda}_{i}(t) = -\frac{\partial H}{\partial x_{i}} = \alpha_{i} \lambda_{i}(t), \qquad i=1,2,\hdots,N, \label{FOOC12i}\\
&\dot{\lambda}_{N+1}(t) = -\frac{\partial H}{\partial t} = -\frac{P}{\eta} \frac{\partial\widehat{\pi}}{\partial t} \sum_{i=1}^{N} u_{i}(t) -  \frac{\partial\widehat{\theta}_{a}}{\partial t}\sum_{i=1}^{N}\alpha_{i}\lambda_{i}(t),\label{FOOC12ii}\\
&\dot{\lambda}_{N+2}(t) = -\frac{\partial H}{\partial x_{N+2}} = 0 \Rightarrow \lambda_{N+2} = \text{constant}.
\label{FOOC12iii}	
\end{align}
%\end{eqnarray}
The transversality condition yields
\begin{eqnarray}
-\displaystyle\sum_{i=1}^{N}\lambda_{i}(T) \mathrm{d}x_{i}(T) - \lambda_{N+1}(T) \mathrm{d}x_{N+1}(T) - \lambda_{N+2}(T) \mathrm{d}x_{N+2}(T) \: + \: H(T) \mathrm{d}T = 0.
\label{transversality}	
\end{eqnarray}
Since the terminal states $\theta_{i}(T)$ are free,  $\mathrm{d}x_{i}(T) = \mathrm{d}\theta_{i}(T) \neq 0 $, and hence (\ref{transversality}) implies that $ \lambda_{i}(T) = 0$, for all $i=1,\hdots,N$. Because $T$ is fixed, $\mathrm{d}x_{N+1}(T) = \mathrm{d}T = 0$. Similarly, $\mathrm{d}x_{N+2}(T) = \frac{\mathrm{d}\tau}{N} = 0$. Combining $\lambda_{i}(T)=0$ with (\ref{FOOC12i}) gives $\lambda_{i}(t) = 0$ for all $t$. Setting $\lambda_{i}(t) \equiv 0$ in (\ref{Hamiltonian12}), and invoking PMP yields the optimal controls as
\begin{eqnarray}
\underset{\left(u_{1}(t), \hdots, u_{N}(t)\right)\in\{0,1\}^{N}}{\text{argmin}}\quad\left(\displaystyle\frac{P}{\eta}\widehat{\pi}(t) + \lambda_{N+2}\right)\displaystyle\sum_{i=1}^{N}u_{i}(t).
\label{EquivMinimize12}
\end{eqnarray}
Hence, if $\frac{P}{\eta}\widehat{\pi}(t) + \lambda_{N+2} > (<) 0$ at any time $t$, then we need to minimize (maximize) $\sum_{i=1}^{N}u_{i}(t)$ over $\{0,1\}^{N}$ at that time, meaning that the \emph{optimal controls are synchronized}.

\item[(ii)] We know that $P,\eta,\widehat{\pi}(t) > 0$ $\forall\:t\in[0,T]$. Thus, if $\lambda_{N+2} \geq 0$, then $\frac{P}{\eta}\widehat{\pi}(t) + \lambda_{N+2} > 0 $ implying $ u_{i}^{*}(t) = 0 \: \forall i=1,\hdots,N,\:\forall t\in[0,T] \: $. This in turn leads to $ \dot{x}_{N+2} = 0 $ implying $ x_{N+2}(t) = x_{N+2}(0) = 0 = x_{N+2}(T) = \frac{\tau}{N}$, which is impossible since $\tau \neq 0$ (given). Therefore, the constant $\lambda_{N+2} < 0$.

Notice that whether $\frac{P}{\eta}\widehat{\pi}(t) + \lambda_{N+2}$ is $>0$ or $<0$ depends on the magnitude of the constant $\lambda_{N+2} < 0$, as well as on the magnitude of $\widehat{\pi}(t) > 0$. Depending on the sign of the time-varying sum $\frac{P}{\eta}\widehat{\pi}(t) + \lambda_{N+2}$, the optimal control will switch between 0 and 1.

Let us denote the optimal value of $\lambda_{N+2}$ as $\lambda_{N+2}^{*}$, and consider a set $S\subseteq [0,T]$ given by $S:=\{s \in [0,T] : \widehat{\pi}(s) < -\frac{\eta}{P}\lambda_{N+2}^{*}\}$. Then, from PMP, $\forall\:i=1,\hdots,N$, we can rewrite the optimal control as
$u_{i}^{*}(t) = 1 \forall t \in S$, and $=0$ otherwise. The statement follows by letting $\pi^{*} := -\frac{\eta}{P}\lambda_{N+2}^{*} > 0$, wherein the uniqueness of $\pi^{*}$ follows from the continuity of $\widehat{\pi}(t)$ (as per assumption in Section \ref{assumptions}).

\item[(iii)] To determine $u_{i}^{*}(t)$, all that remains is to determine $\lambda_{N+2}^{*}$, or equivalently $-\frac{\eta}{P}\lambda_{N+2}^{*}$. The choice of $\lambda_{N+2} < 0$, or equivalently $-\frac{\eta}{P}\lambda_{N+2} > 0$, is constrained by the terminal condition
\begin{eqnarray*}
x_{N+2}(T) = \frac{\tau}{N} \quad\Leftrightarrow\quad \int_{0}^{T}\sum_{i=1}^{N} u_{i}(t) \: \mathrm{d}t = \frac{\tau}{N} \quad\stackrel{\text{PMP}}{\Longrightarrow}\quad \int_{0}^{T} \sum_{i=1}^{N}\mathds{1}_{\{\widehat{\pi}(t) < -\frac{\eta}{P}\lambda_{N+2}\}} \: \mathrm{d}t = \frac{\tau}{N},	
\end{eqnarray*}
and hence feasible values of $-\frac{\eta}{P}\lambda_{N+2}$ comprise the set $\big\{-\frac{\eta}{P}\lambda_{N+2} \in \mathbb{R}^{+} : \int_{0}^{T} \mathds{1}_{\{\widehat{\pi}(t) < -\frac{\eta}{P}\lambda_{N+2}\}} \: \mathrm{d}t = \frac{\tau}{N}\big\}$. The optimal $\lambda_{N+2}^{*}$, that minimizes the ``cost-to-go'' $\frac{P}{\eta}\int_{0}^{T}\widehat{\pi}(t)\sum_{i=1}^{N}\mathds{1}_{\{\widehat{\pi}(t) < -\frac{\eta}{P}\lambda_{N+2}\}} \: \mathrm{d}t$, is given by
\begin{eqnarray*}
	-\frac{\eta}{P}\lambda_{N+2}^{*} = \inf \Bigg\{-\frac{\eta}{P}\lambda_{N+2}\in\mathbb{R}^{+} :  \int_{0}^{T} \mathds{1}_{\{\widehat{\pi}(t) < -\frac{\eta}{P}\lambda_{N+2}\}} \: \mathrm{d}t = \frac{\tau}{N}\Bigg\}.
\end{eqnarray*}

To determine $\lambda_{N+1}^{*}(t)$, combining (\ref{FOOC12ii}) with $\lambda_{i}(t) = 0$ results in
\begin{eqnarray*}
\dot{\lambda}_{N+1} = -\frac{P}{\eta}N\mathds{1}_{\{\widehat{\pi}(t) < -\frac{\eta}{P}\lambda_{N+2}^{*}\}}\frac{\partial \widehat{\pi}}{\partial t}.	
\end{eqnarray*}
Thus, $\lambda_{N+1}\rvert_{u_{i}^{*}(t)=0} = H\rvert_{u_{i}^{*}(t)=0} =$ constant, which we enforce to be zero. On the other hand, $\lambda_{N+1}\rvert_{u_{i}^{*}(t)=1} = - \frac{P}{\eta} N \widehat{\pi}(t) + k$, where the integration constant $k$ needs to be determined. Since the Hamiltonian evaluated at $u_{i}^{*}(t)=1$, is
\begin{eqnarray*}
	H\rvert_{u_{i}^{*}(t)=1} = N\left(\frac{P}{\eta}\widehat{\pi}(t) + \lambda_{N+2}^{*}\right) + \lambda_{N+1}\rvert_{u_{i}^{*}(t)=1} = N\lambda_{N+2}^{*} + k = \;\text{constant},
\end{eqnarray*}
which, as before, we enforce to be zero, we obtain
\begin{eqnarray*}
	k=-N\lambda_{N+2}^{*} \Rightarrow \lambda_{N+1}\rvert_{u_{i}^{*}(t)=1} = -\frac{P}{\eta}N\widehat{\pi}(t) - N\lambda_{N+2}^{*} = \frac{P}{\eta}N\left(\pi^{*} - \widehat{\pi}(t)\right) > 0
\end{eqnarray*}
as $\pi^{*} > \widehat{\pi}(t)$ $\forall t \in S$.

To derive $x_{i}^{*}(t)$ for all $i=1,\hdots,N$, we simply substitute the optimal control $u_{i}^{*}(t)$ into (\ref{DynamicsConstraintInitial}), and then integrate the resulting first-order linear non-homogeneous ODE using the method of integrating factor, yielding the desired expression.

\end{enumerate}

\end{proof}

\begin{figure}[t]
  \centering
    \includegraphics[width=0.68\textwidth]{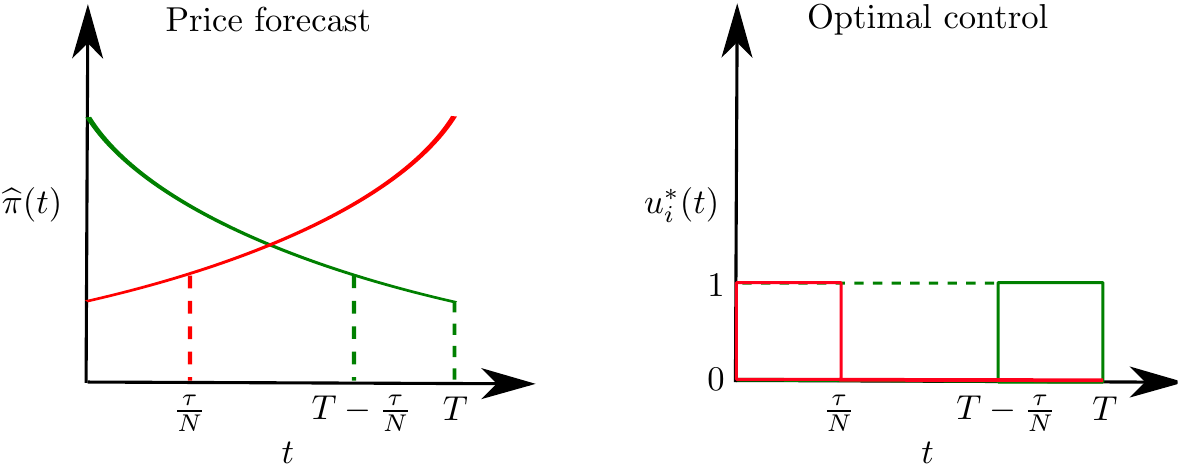}
    \caption{Left: Strictly increasing (\emph{red}) and decreasing (\emph{green}) price forecasts $\widehat{\pi}(t)$, Right: the corresponding optimal control with constraint \textbf{C3} inactive is given by $u_{i}^{*}(t) = \mathds{1}_{S}$, where $S = [0,\frac{\tau}{N}]$ for strictly increasing, and $S=[T-\frac{\tau}{N},T]$ for strictly decreasing $\widehat{\pi}(t)$.}
    \label{Monotone12}
    %\vspace*{-0.21in}
\end{figure}

\begin{remark}\label{MonotoneRearrangmentRemark}
\textbf{Monotone rearrangement of $\widehat{\pi}(t)$:}
The main insight behind the optimal control derived in Theorem \ref{12theorem} can be obtained by looking at strictly monotone price forecasts, as shown in Fig. \ref{Monotone12}\!\!\!. In these cases, it is intuitive that the optimal ON periods of the TCLs lie at either end of the interval $[0,T]$. Theorem \ref{12theorem} tells us that the same insight can be extended to non-monotone $\widehat{\pi}(t)$, by first computing its monotone rearrangement (p. 276, Ch. 10 in Hardy \etal. \cite{HardyLittlewoodPolyaBook}) $\widehat{\pi}^{\uparrow}(t)$, and then computing the threshold $\pi^{*}$ and the ON time set $S$ from this monotone rearrangement as a function of $\frac{\tau}{N}$, as illustrated in Fig. \ref{MonotoneRearrangement}\!\!\!. This is especially relevant noting that the typical $\widehat{\pi}(t)$ is non-monotone and looks as in Fig. \ref{DApriceforecast}\!\!\!.
\end{remark}

\begin{figure*}[t]
  \centering
    \includegraphics[width=0.95\textwidth]{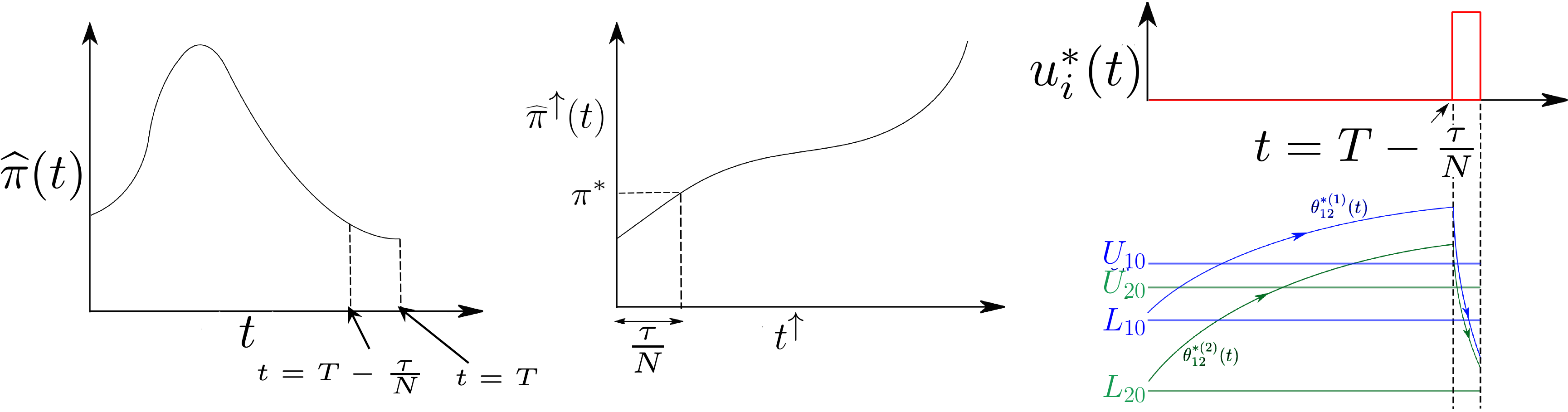}
    \caption{Left: A non-monotone price forecast $\widehat{\pi}(t)$ with $S=[T-\frac{\tau}{N},T]$ being the subset of times in $[0,T]$, with measure $\frac{\tau}{N}$, that corresponds to minimum price. Middle: $\widehat{\pi}^{\uparrow}(t)$ is the increasing rearrangement of the function $\widehat{\pi}(t)$, plotted against the
    corresponding re-arranged time $t^{\uparrow}$. Right: With constraint \textbf{C3} inactive, the optimal control $u_{i}^{*}(t)$ and optimal indoor temperature trajectories $\theta_{12}^{*(i)}(t)$ are shown for $i=1,2$ TCLs. The subscript $_{12}$ in $\theta_{12}^{*(i)}(t)$ denotes that constraints \textbf{C1} and \textbf{C2} are active.}
    \label{MonotoneRearrangement}
    %\vspace*{-0.15in}
\end{figure*}

\begin{remark}\label{NonuniquenessRemark}
\textbf{Non-uniqueness of optimal control:}
From Theorem \ref{12theorem}, the synchronized optimal control $\{u_{i}^{*}(t)\}_{i=1}^{N}$, is unique iff the set $S$ is unique, where $S$ is the pre-image of $\pi^{*}$. Notice that although $\pi^{*}$ is unique for any continuous $\widehat{\pi}(t)$, uniqueness of $S$ depends on whether there exist time intervals of constancy in price forecast $\widehat{\pi}(t)$. For example, in Fig. \ref{MonotoneRearrangement}\!\!\!, there is no such interval of constancy, and hence the pre-image set $S$, and the optimal control $u_{i}^{*}(t)$, are unique. This remains true even when $\frac{\tau}{N}$ is large (see Fig. \ref{nonunique}\!\!\!(a)). Non-uniqueness, however, can arise if there exist an interval of constancy $V$, and $\frac{\tau}{N}$ is large enough that $S$ contains at least a subset of $V$ (see Fig. \ref{nonunique}\!\!\!(b)). The uncountable number of non-unique solutions arising from such a situation can be resolved by fixing the convention of choosing the optimal control with minimum number of switchings.
\end{remark}

%\subsection{When Constraint C3 is not Hit}
%\label{HowNotToHit}

\begin{figure}[t]
%\vspace*{-0.1in}
  \centering
    \includegraphics[width=0.85\textwidth]{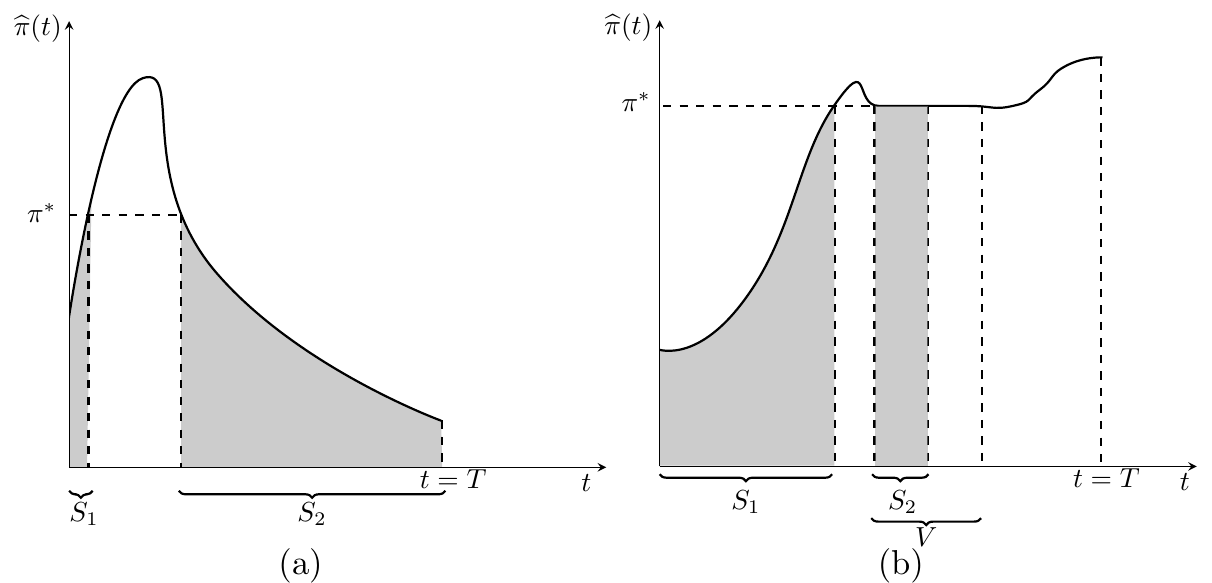}
    \caption{(a) For this non-monotone price forecast, the set $S = S_{1} \cup S_{2}$ is unique, where $|S| = \frac{\tau}{N}$, and hence $u_{i}^{*}(t)$ is unique too. (b) Here also $S = S_{1} \cup S_{2}$ such that $|S| = \frac{\tau}{N}$, however, $\frac{\tau}{N}$ and hence the threshold $\pi^{*}$ are such that the set $S_{2} \subset V$, where $V$ is the total interval of constancy. Thus, subsets of $S_{2}$ can be interchanged with equi-measure subsets of $V\setminus S_{2}$ without affecting the cost. Therefore, the set $S$ and optimal control $u_{i}^{*}(t)$ are not unique in this case. Since these sets are intervals, the number of such interchanges, and hence the number of optimal controls, is uncountable.}
    \label{nonunique}
    %\vspace*{-0.2in}
\end{figure}

\subsection{Solution with Zero Amplitude Chattering}
\label{ZeroAmpSoln}
Now we focus on solving (\ref{PlanningCostFcn}) subject to constraints (\textbf{C1})--(\textbf{C3}), under the assumption that the indoor temperature trajectories $\theta_{i}(t)$ can slide along the boundaries $L_{i0}$ and $U_{i0}$, which can be thought of as the limits of small amplitude chattering. We assume that sliding along $L_{i0}$ holds the ON mode, while the same along $U_{i0}$ holds the OFF mode. Our objective is to obtain qualitative insight into the solution structure under these simplifying assumptions. In Section \ref{minChatteringAmplabel}, we consider the practical case of finite amplitude chattering.

To describe the optimal solution, we next define the two-sided Skorokhod map \cite{Kruk2007,Kruk2008}, which generalizes the one-sided version originally introduced by Skorokhod \cite{Skorokhod1961}.

\begin{definition} \textbf{Two-sided Skorokhod Map:}\\
	Given $0<L<U<\infty$, and scalar trajectory $y(\cdot)\in\mathbb{D}\left((-\infty,\infty)\right)$, the two-sided Skorokhod map $\Psi_{L,U}: \mathbb{D}\left((-\infty,\infty)\right) \mapsto \mathbb{D}\left([L,U]\right)$ is defined as $z(t) = \Psi_{L,U} (y(t)) := \Lambda_{L,U} \circ \Psi_{L,\infty} (y(t))$, where
\begin{align*}
\Lambda_{L,U}\left(\phi\right)\left(t\right) &:= \phi\left(t\right) - \displaystyle\sup_{0\leq s \leq t} \left(\left[\phi\left(s\right) - U\right]^{+} \wedge \displaystyle\inf_{s\leq r \leq t}\left(\phi\left(r\right) - L\right)\right),\\[1em]
\Psi_{L,\infty} (y(t)) &:= y(t) + \displaystyle\sup_{0\leq s \leq t} \left[L - y(s)\right]^{+}.	
\end{align*}
\label{TwoSidedSkorokhodMapDefn}
\end{definition}

The following Theorem summarizes the solution for problem (\ref{PlanningCostFcn}) with constraints \textbf{C1}--\textbf{C3}, under the simplifying assumption of zero amplitude chattering. The optimal controls are shown to be identical to those in Theorem \ref{12theorem}. Interestingly, it is shown that the optimal indoor temperature trajectories in this case can be obtained by applying the two-sided Skorokhod maps on the optimal indoor temperature trajectories obtained from Theorem \ref{12theorem}. Here, the Skorokhod maps are parameterized by the upper and lower comfort boundaries of the individual TCLs.

\begin{theorem}
\label{123theorem}
Consider problem (\ref{PlanningCostFcn}) with constraints \textbf{C1}--\textbf{C3}. The optimal controls are synchronizing, and, as in Theorem \ref{12theorem}, based on a price forecast threshold $\pi^{*}$, switch between 0 and 1. Assuming zero amplitude chattering to be feasible, the open loop optimal controls $u_{i}^{*}(t)$ are identical to those in Theorem \ref{12theorem}. For $i=1,\hdots,N$, the optimal states $x_{i}^{*}(t)$ are the two-sided Skorokhod maps parameterized by individual comfort ranges $[L_{i0},U_{i0}]$, acting on respective optimal states from Theorem \ref{12theorem}, i.e., $\theta_{123}^{*(i)}(t) = \Psi_{L_{i0},U_{i0}}\left(\theta_{12}^{*(i)}(t)\right)$, where $\theta_{123}^{*(i)}(t)$ is the optimal indoor temperature trajectory when constraints \textbf{C1}--\textbf{C3} are active, and $\theta_{12}^{*(i)}(t)$ is the same when constraints \textbf{C1}--\textbf{C2} are active, for the $i$\textsuperscript{th} TCL.
\end{theorem}

\begin{proof}
From the necessary conditions for optimality under state inequality constraints \cite{bryson1963optimal,VinterBook}, it can be directly verified that the $\pi^{*}$ is as in Theorem \ref{12theorem}. Hence the monotone rearrangement argument applies as before, and $u_{i}^{*}(t)$ are synchronized as function of time. However, the optimal states have different hitting times to the respective boundaries $L_{i0}$ and $U_{i0}$. For brevity, we provide below a simple graphical argument for the optimal states for strictly monotone (w.l.o.g. decreasing) $\widehat{\pi}(t)$. The non-monotone $\widehat{\pi}(t)$ can be dealt via the monotone rearrangement, and the non-uniqueness due to constancy can be dealt with by adopting a minimum switching convention, as earlier.

By the argument above, consider strictly decreasing $\widehat{\pi}(t)$ as in Fig. \ref{Monotone12}\!\!\! left, green curve, and fix the $i$\textsuperscript{th} home with initial indoor temperature $\theta_{i0}$, and comfort boundaries $L_{i0}$ and $U_{i0}$. At time $t=0^{+}$, the trajectory $\theta_{i}(t)$ can move either exponentially upward or downward. For $t\in[0,T]$, let us call the set of all feasible trajectories $\theta_{i}(t) \in [L_{i0}, U_{i0}]$ for which $u_{i}(0^{+})=0$, as the ``initially up-going family". Similarly, define ``initially down-going family" for $u_{i}(0^{+})=1$. Our proof consists of the following two steps.

Step 1: Finding the optimal indoor temperature trajectory among the ``initially up-going family": We notice that among the ``initially up-going family'', it is optimal to hit $U_{i0}$, since otherwise turning the TCL ON before hitting $U_{i0}$ strictly increases the cost, as $\widehat{\pi}(t)$ is strictly decreasing. Similarly, starting from the time at which $U_{i0}$ is hit, it is then optimal to hold till $T-\frac{\tau}{N}$ as sliding along $U_{i0}$, as per assumption, does not contribute to the cost. For $t\in[T-\frac{\tau}{N},T]$, we notice that we must keep $u_{i}(t) = 1$ to respect the energy constraint. Further, notice that any de-synchronization among TCLs increase cost. Thus, the optimal temperature trajectory among the ``initially up-going family" looks like those shown in Fig. \ref{SkorokhodNonmonotone}\!\!\! bottom right. From Definition \ref{TwoSidedSkorokhodMapDefn}, we find that the optimal indoor temperature trajectory among the ``initially up-going family'' is $\Psi_{L_{i0},U_{i0}}\left(\theta_{12}^{*(i)}(t)\right)$.

Step 2: Showing that any trajectory from the ``initially down-going family" has cost strictly larger than the same for the optimal trajectory in Step 1: This can be easily verified by comparing the optimal from Step 1, with any trajectory from the ``initially down-going family" using that $\widehat{\pi}(t)$ is strictly decreasing.

Combining the above two steps, we conclude that for $\widehat{\pi}(t)$ strictly decreasing, the optimal indoor temperature trajectory found in Step 1 is the optimal among all feasible indoor temperature trajectories. Similar argument applies to $\widehat{\pi}(t)$ strictly increasing, and to monotone rearranged version $\widehat{\pi}^{\uparrow}(t)$ in case $\widehat{\pi}(t)$ is non-monotone. We eschew the details and illustrate an example in Fig. \ref{SkorokhodNonmonotone}\!\!\! to help the readers follow our main argument.
\end{proof}

\begin{figure}[htpb]
  \centering
    \includegraphics[width=0.68\textwidth]{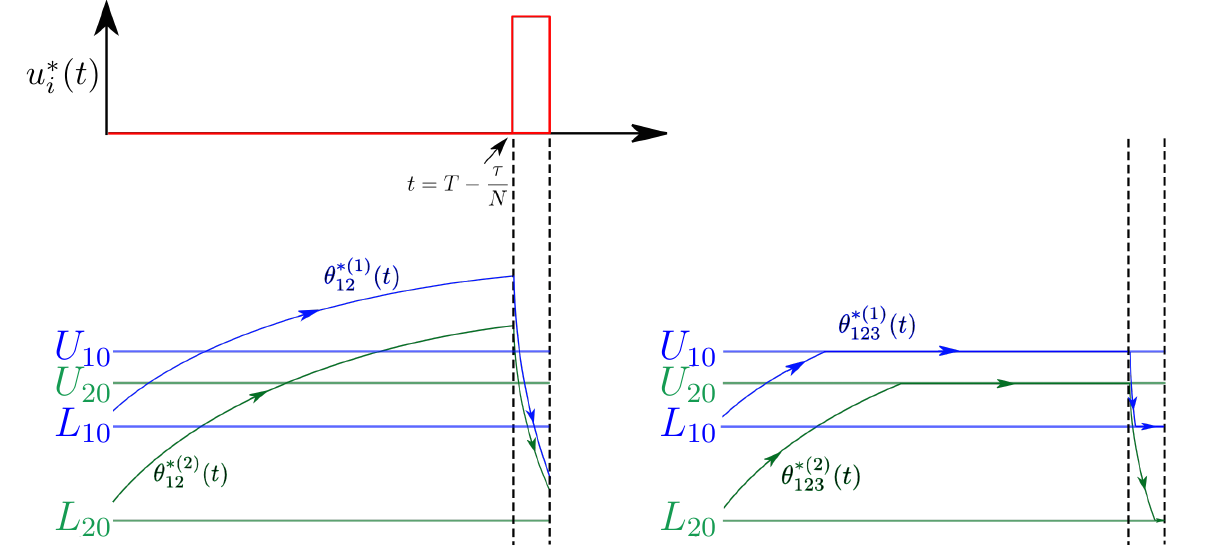}
    %\vspace*{-0.1in}
    \caption{For the non-monotone price forecast $\widehat{\pi}(t)$ same as in Fig. \ref{MonotoneRearrangement}\!\!\!, the optimal controls $u_{i}^{*}(t)$ (top left) with constraints \textbf{C1}--\textbf{C3}, are same as those shown in Fig. \ref{MonotoneRearrangement}\!\!\!. However, the optimal indoor temperatures $\theta_{123}^{*(i)}(t)$ (right bottom) are different from $\theta_{12}^{*(i)}(t)$ (left bottom, same as in Fig. \ref{MonotoneRearrangement}\!\!\!), for $i=1,2$ TCLs. The subscript $_{123}$ in $\theta_{123}^{*(i)}(t)$ denotes that constraints \textbf{C1}--\textbf{C3} are active.}
    \label{SkorokhodNonmonotone}
    %\vspace*{-0.2in}
\end{figure}

%%%%%%%%%%%%%%%%%%%%%%%%%%%%%%%%%%%%%%%%%%%%%%%%%%%%%%%%%%%%%%%%%%%%%%%%%%%%%%%%

\section{Implementable Solution with A Specified Minimum Switching Period}
\label{minChatteringAmplabel}

%\subsection{Monotone Price}
%Duty cycle formula as theorem

%\subsection{Non-monotone Price}
%Structural theorem of Fernando

Since physical thermostats have a minimum chattering amplitude, or equivalently minimum ON-OFF time period $T_{m} > 0$, it is important to find an algorithm that explicitly accounts for this device limitation in the control design. Given this parameter $T_{m}$, the following Theorem gives an \emph{exact} algorithm to compute the binary optimal control via convexification.

The importance of the Theorem below lies in decoupling the two technical difficulties in solving problem (\ref{PlanningCostFcn}) subject to (\textbf{C1})--(\textbf{C3}), namely the non-convexity of the control and the presence of state inequality constraints. In other words, it allows us to first solve the convexified optimal control problem with the state inequality constraints, and then use Algorithm \ref{vstar2ustar} (introduced as part of the Theorem below) as a post-processing tool to recover the binary optimal controls respecting the prescribed minimum ON-OFF time period $T_{m}$.

\begin{theorem}
\label{RecoverOptimalBinary}
For $i=1,\hdots,N$, consider the control convexification $u_{i}(t)\in \{0,1\} \mapsto v_{i}(t) \in [0,1]$, and let $v_{i}^{*}(t)\in[0,1]$ be the optimal convexified control corresponding to the non-convex optimal control problem (\ref{PlanningCostFcn})--(\ref{ComfortRangeConstraint}) with $N$ TCLs having thermal coefficients $\{\alpha_{i},\beta_{i}\}_{i=1}^{N}$. Let $\vartheta_{i}^{*}(t)$ (resp. $\theta_{i}^{*}(t)$) be the indoor temperature trajectory realized by the optimal control $v_{i}^{*}(t)$ (resp. $u_{i}^{*}(t)$). Then, Algorithm \ref{vstar2ustar} recovers the optimal control $u_{i}^{*}(t)\in\{0,1\}$ from $v_{i}^{*}(t)\in[0,1]$, while guaranteeing that the indoor temperatures trajectories $\vartheta_{i}^{*}(t)$ and $\theta_{i}^{*}(t)$ coincide at the end of each minimum allowable time period of length $T_{m}$.
%\begin{mdframed}
\begin{algorithm}
\caption{Recovering $u_{i}^{*}(t)\in\{0,1\}$ from $v_{i}^{*}(t)\in[0,1]$}
\label{vstar2ustar}
\begin{algorithmic}[1]
\If {$v_{i}^{*}(t) = 0 \;\text{OR}\; 1$}
    \State $u_{i}^{*}(t) = v_{i}^{*}(t)$
\Else
    \If {$\vartheta_{i}^{*}(t)$ is at upper boundary}
        \State $u_{i}^{*}(t) = \begin{cases}1 & \forall\:t\in[(j-1)T_{m}, (j-1)T_{m}+\overline{\gamma}_{i})\\
 0	& \forall\:t\in[(j-1)T_{m}+\overline{\gamma}_{i}, j T_{m})
 \end{cases}
$\\
\qquad\quad where $j=1,\hdots,\bigg\lceil \frac{\text{spt}(v_{i}^{*}(t))}{T_{m}} \bigg\rceil$, $t\in[0,T]$
	\Else{\:$\vartheta_{i}^{*}(t)$ is at lower boundary}
        \State $u_{i}^{*}(t) = \begin{cases}0 & \forall\:t\in[(j-1)T_{m}, jT_{m}-\underline{\gamma}_{i})\\
 1	& \forall\:t\in[jT_{m}-\underline{\gamma}_{i}, j T_{m})
 \end{cases}
$\\
\qquad\quad where $j=1,\hdots,\bigg\lceil \frac{\text{spt}(v_{i}^{*}(t))}{T_{m}} \bigg\rceil$, $t\in[0,T]$	
   \EndIf
\EndIf
\end{algorithmic}
\end{algorithm}
%\end{mdframed}

\noindent In Algorithm \ref{vstar2ustar}, $\{(\overline{\gamma}_{i},\underline{\gamma}_{i})\}_{i=1}^{N}$ are time duration pairs such that the binary optimal trajectory $u_{i}^{*}(t)$ consists of two duty cycles: $\overline{\gamma_{i}}/T_{m}$ and $\underline{\gamma_{i}}/T_{m}$, where
\begin{align}
\overline{\gamma}_{i} &= \frac{1}{\alpha_{i}}\log\left(1 + \alpha_{i}\int_{0}^{T_{m}}e^{\alpha_{i} s}v_{i}^{*}(s)\:\mathrm{d}s\right),
\label{gammaUpper}\\
\underline{\gamma}_{i} &= \frac{1}{\alpha_{i}}\log\left(\displaystyle\frac{1}{1 - \alpha_{i}e^{-\alpha_{i}T_{m}}\int_{0}^{T_{m}}e^{\alpha_{i} s}v_{i}^{*}(s)\:\mathrm{d}s}\right).
\label{gammaLower}	
\end{align}

\end{theorem}

\begin{proof}
Since $v_{i}^{*}(t)$ is binary iff the respective upper and lower comfort boundaries are not hit, hence $u_{i}^{*}(t) = v_{i}^{*}(t)$ when $v_{i}^{*}(t) = 0$ or $1$. We know that $v_{i}^{*}(t) \in (0,1)$ iff $\vartheta_{i}^{*}(t)$ is at either upper or lower boundary. Clearly, at the upper (lower) boundary, such cycles should begin with an ON (OFF) segment, and end with an OFF (ON) segment. Matching indoor temperature values at each end of these switching period means $\vartheta_{i}^{*}(T_{m}) = \theta_{i}^{*}(T_{m})$, which gives
\begin{eqnarray}
\int_{0}^{T_{m}}e^{\alpha_{i} s} v_{i}^{*}(s)\:\mathrm{d}s = \int_{0}^{T_{m}}e^{\alpha_{i} s} u_{i}^{*}(s)\:\mathrm{d}s.
\label{MatchingCondition}
\end{eqnarray}
At upper boundary, RHS of (\ref{MatchingCondition}) equals $\int_{0}^{\overline{\gamma}_{i}}e^{\alpha_{i}s}\:\mathrm{d}s$, which solved for $\overline{\gamma}_{i}$ yields (\ref{gammaUpper}). At lower boundary, RHS of (\ref{MatchingCondition}) equals $\int_{T_{m} - \underline{\gamma}_{i}}^{T_{m}}e^{\alpha_{i}s}\:\mathrm{d}s$, which solved for $\underline{\gamma}_{i}$ yields (\ref{gammaLower}).  	
\end{proof}

%%%%%%%%%%%%%%%%%%%%%%%%%%%%%%%%%%%%%%%%%%%%%%%%%%%%%%%%%%%%%%%%%%%%%%%%

It should be noted that although the indoor temperature trajectories $\theta_i^*$ and $\vartheta_i^*$ (corresponding to the controls $u_i^{*}$ and $v_i^{*}$, respectively) periodically coincide at the end of each time period $T_m$, the cost of each solution is not necessarily the same.
While the solution corresponding to  $v_i^{*}$ is the theoretical, though non-implementable, optimal solution, the solution corresponding to  $u_i^{*}$ is a suboptimal solution that has the closest implementable trajectory.

\begin{remark}
	Our optimal control problem can be viewed as an optimal control problem of a switched system; see e.g., \cite{bengea2005}.
 The use of an approximate relaxed version of the problem with a convexified control set, as is done here, has been extensively studied  in the literature (see \cite{berkovitz2012} and the references therein).
 More recently, a so-called embedding principle has been investigated in the works \cite{bengea2005,vasudevan2013,chen2017}.
 In those works, a relaxed version of the optimal control problem is first solved in a convexified input set, and then a projection operator is used to obtain the input in the original discrete set.
   This conceptual path is also followed here in Theorem 3.
   Nevertheless, the goal with which the techniques are used is different in our approach.
In the references mentioned, one of the main concerns is to address the limitation that the switched optimal control problem, without imposing additional assumptions related to the possibility of chattering, only has a solution when the space of controls is enlarged to the space of relaxed controls. The goal then becomes to construct approximate solutions that are consistent, in the sense that in the limit they converge to the optimal solution.
In our case, the starting point is a physical limitation of the system imposing a minimum ON-OFF time period of $T_m>0$. This physical limitation itself prevents infinite frequency chattering, and the problem becomes to construct solutions to an approximate problem such that they can be easily projected into a physical realizable solution space.
\end{remark}  

%%%%%%%%%%%%%%%%%%%%%%%%%%%%%%%%%%%%%%%%%%%%%%%%%%%%%%%%%%%%%%%%%%%%%%%%%%%%%%%%

\section{Numerical Simulation}
\label{NumericalSimulationSection}
To illustrate how the LSE can use the results derived so far for the purpose of demand response, we now provide a numerical example where the LSE computes the day-ahead minimum cost energy procurement for its $N=500$ customers' TCLs based on ERCOT day-head price forecast $(\widehat{\pi}(t))$ data as shown in Fig. \ref{DApriceforecast}\!\!\!, and the ambient temperature forecast $(\widehat{\theta}_{a}(t))$ data for the same day (August 10, 2015) available from a weather station in Houston, Texas. These forecast data and the  real-time ambient temperature $(\theta_{a}(t))$ data on August 11, 2015, are shown in Fig. \ref{PriceAmbient}\!\!\!.
\begin{figure}[t]
  \centering
%  \vspace*{-0.1in}
    \includegraphics[width=0.68\textwidth]{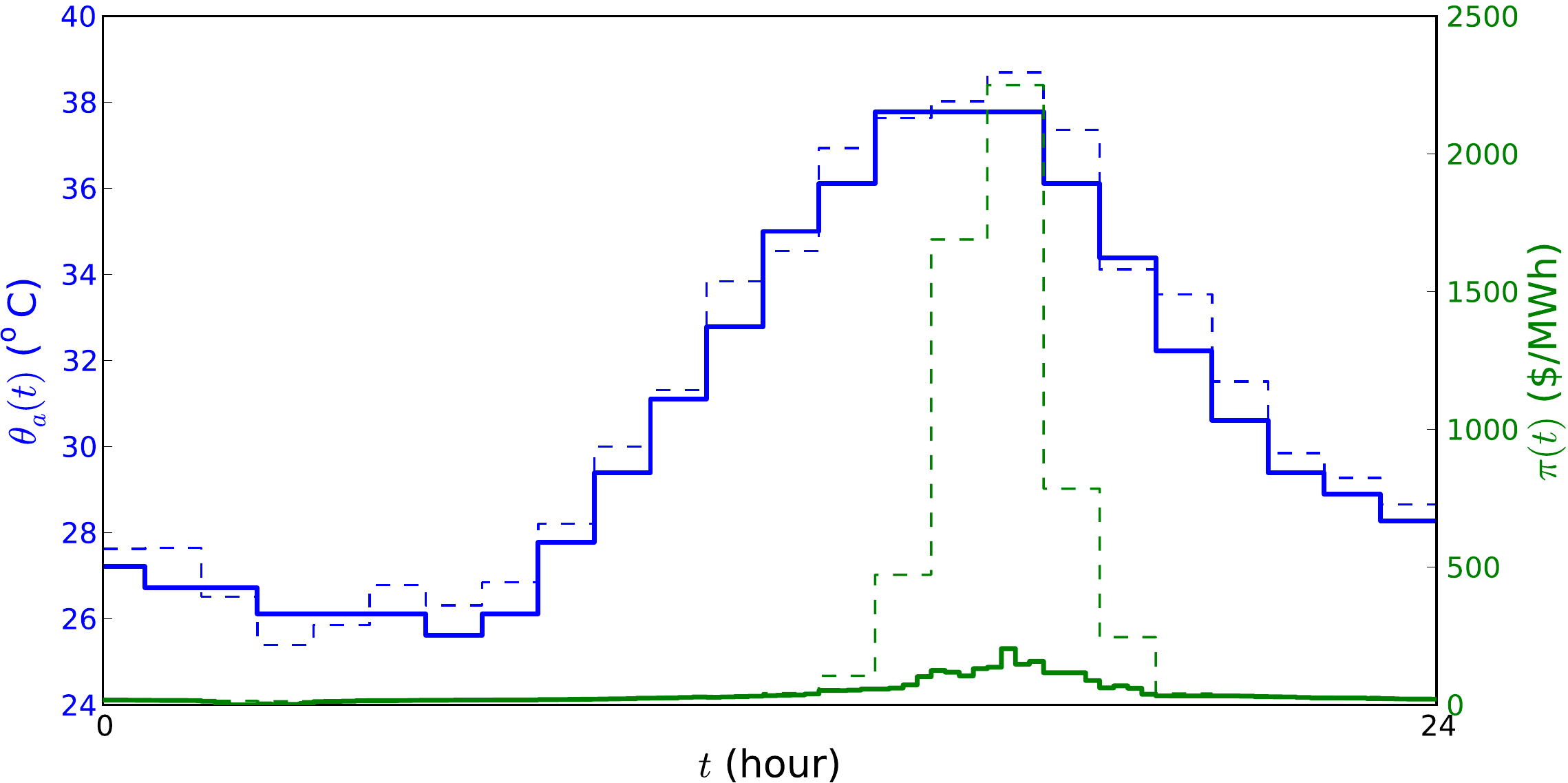}
    \caption{The day-ahead ambient temperature forecast $\widehat{\theta}_{a}(t)$ (\emph{dashed blue}), real-time ambient temperature $\theta_{a}(t)$ (\emph{solid blue}), ERCOT day-ahead price $\widehat{\pi}(t)$ (\emph{dashed green}), and ERCOT real-time price  $\pi(t)$ (\emph{solid green, not used for computation in this paper}) data for Aug. 11, 2015 in Houston.}
    \label{PriceAmbient}
    %\vspace*{-0.02in}
\end{figure}

\begin{figure}[h]
  \centering
    \includegraphics[width=0.68\textwidth]{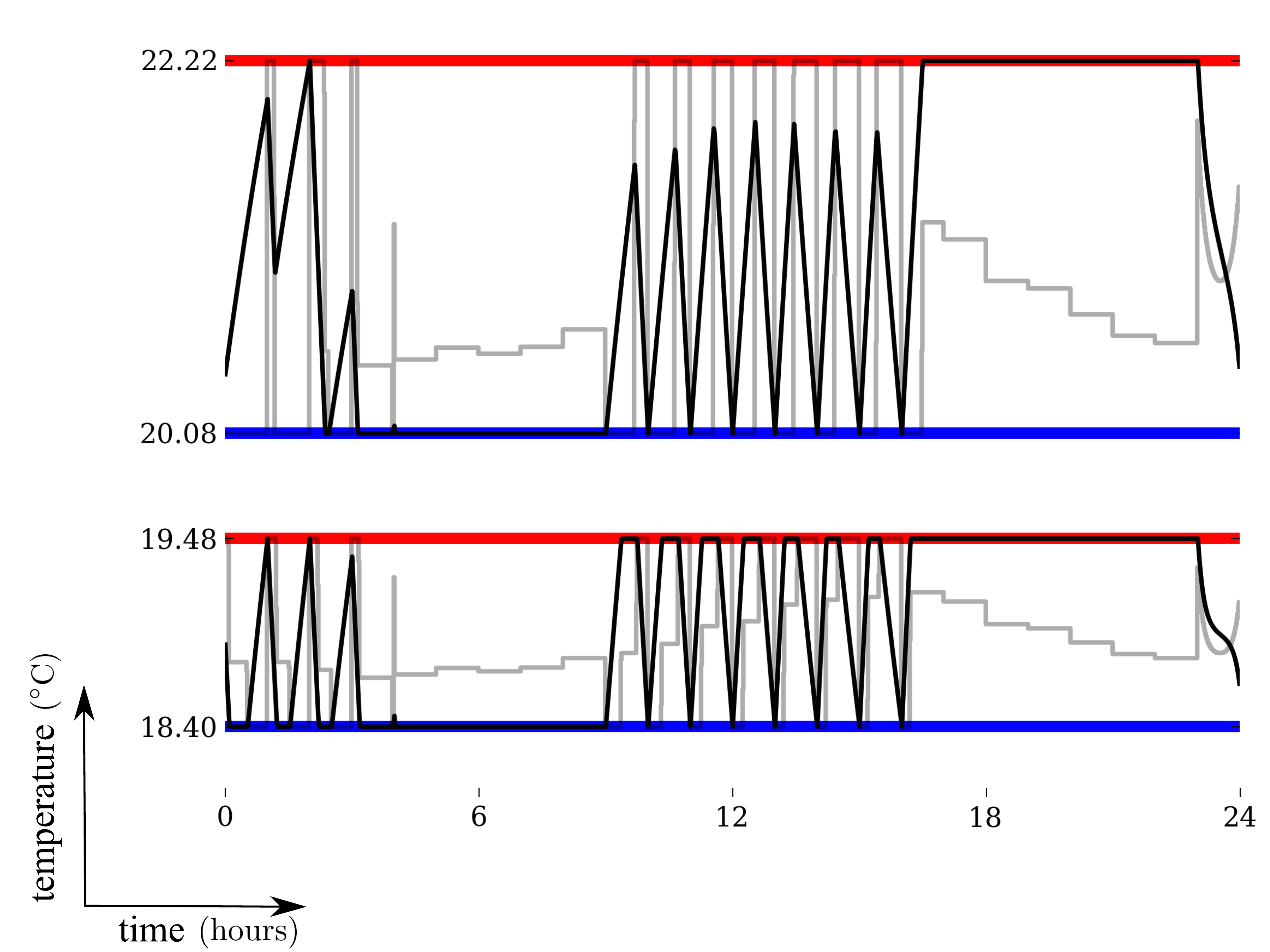}
    \caption{Convexified optimal control $v_{i}^{*}(t)\in[0,1]$ (\emph{gray}) and corresponding indoor temperature $\vartheta_{i}^{*}(t)$ in $^{\circ}\text{C}$ (\emph{black}), $i=1,2$, trajectories for two out of the total population of 500 TCL customers for which the LSE designs the optimal power consumption corresponding to the day-ahead price and ambient temperature forecast shown in Fig. \ref{PriceAmbient}\!\!\!. Details of the simulation setup are described in Section \ref{NumericalSimulationSection}. For the two representative TCLs shown above, their lower (resp. upper) comfort boundaries $L_{i}$ (resp. $U_{i}$) are depicted as \emph{blue} (resp. \emph{red}) thick horizontal lines. Specifically, $[L_{1},U_{1}] \equiv [18.40^{\circ}{\rm{C}}, 19.48^{\circ}{\rm{C}}]$ and $[L_{2},U_{2}] \equiv [20.08^{\circ}{\rm{C}}, 22.22^{\circ}{\rm{C}}]$, as shown. The parameters $(\alpha_{1},\alpha_{2}) = (4.4032, 4.1067)\times 10^{-3}$ seconds$^{-1}$, and $(\beta_{1},\beta_{2}) = (8.4510, 8.5286) \times 10^{-3}$ $\frac{^{\circ}\text{C}}{\text{kW seconds}}$. The optimal control trajectories $v_{i}^{*}(t)\in[0,1]$ (\emph{gray}) are scaled between the respective comfort boundaries, i.e., $v_{i}^{*}=1$ (resp. 0) at the upper (resp. lower) boundary. Notice that $v_{i}^{*}(t)$ takes fractional values whenever $\vartheta_{i}^{*}(t)$ lies at either upper or lower comfort boundary, as mentioned in the proof of Theorem \ref{RecoverOptimalBinary}.}
\label{LPtrajcontrolTwoHomes}
    %\vspace*{-0.02in}
\end{figure}

\begin{figure}[h]
  \centering
    \includegraphics[width=0.48\textwidth]{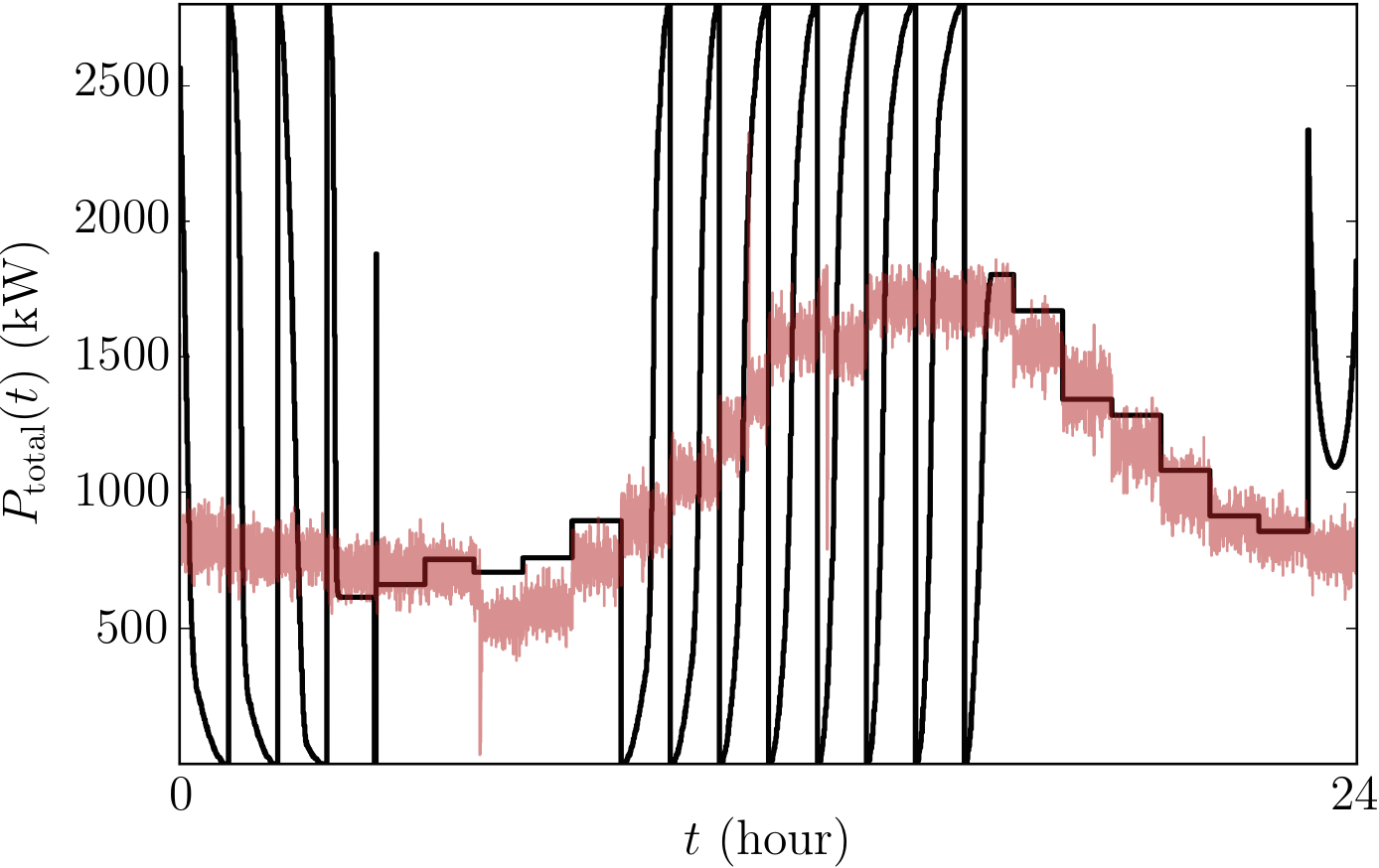}
    \caption{The \emph{black} curve shown above is the optimal power consumption trajectory $P_{\text{total}}^{\text{ref}}(t) = P_{e}\sum_{i=1}^{N}u_{i}^{*}(t)$ computed by solving the optimal control problem (\ref{PlanningCostFcn})--(\ref{ComfortRangeConstraint}), via Algorithm \ref{vstar2ustar} with $T_{m} = 1.5$ minutes. The \emph{brick} colored curve is the real-time controlled aggregate consumption $P_{\text{total}}(t)$ corresponding to control gain tuple $(k_{p},k_{i},k_{d}) = (10^{-4},10^{-6},10^{-4})$ used to move the thermostatic boundaries (see Section III.1 in \cite{AbhishekTPS2017} for details). The LSE can invoke demand response by controlling the setpoints of the TCL population in such a way that their real-time aggregate consumption $P_{\text{total}}(t)$ track the reference aggregate consumption $P_{\text{total}}^{\text{ref}}(t)$. The tracking error between the two curves depend on the forecasted versus real-time ambient temperature mismatch, as well as on the thermal inertia of the TCLs in the population.}
    \label{SetpointTracking}
    %\vspace*{-0.2in}
\end{figure}

With the initial conditions and parameters of the heterogeneous TCL population as in Section V-A in Halder \etal. \cite{AbhishekTPS2017}, $\overline{\tau} = \frac{1}{3}$ (which was verified to be feasible using (\ref{FeasibilitywZeroDynamics})), comfort tolerances $\{\Delta_{i}\}_{i=1}^{N}$ sampled randomly from a uniform distribution over $[0.1^{\circ}\mathrm{C}, 1.1^{\circ}\mathrm{C}]$, and for $\widehat{\pi}(t)$ and $\widehat{\theta}_{a}(t)$ as in Fig. \ref{PriceAmbient}\!\!\!, the LSE solves the optimal control problem (\ref{PlanningCostFcn})--(\ref{ComfortRangeConstraint}) by first convexifying the controls $u_{i}(t) \in \{0,1\} \mapsto v_{i}(t)\in[0,1]$ for $i=1,\hdots,N$, and then recovering the optimal controls $\{u_{i}^{*}\}_{i=1}^{N}$ using Theorem \ref{RecoverOptimalBinary}. For this computation, we used 1 minute time-step for Euler discretization of dynamics (\ref{DynamicsConstraintInitial}), and solved the resulting LP with 1 million 440 thousand decision variables (see Section \ref{DifficultyDNS}) using MATLAB \texttt{linprog}. In Fig. \ref{LPtrajcontrolTwoHomes}\!\!\!, we show the \emph{convexified} optimal controls $v_{i}^{*}(t)$ (gray curves) and corresponding indoor temperature trajectories $\vartheta_{i}^{*}(t)$ (black curves) for two representative TCLs out of the total $N =$ 500 TCLs. This computation was followed by applying Algorithm \ref{vstar2ustar} to evaluate the mapping $(v_{i}^{*}(t),\vartheta_{i}^{*}(t)) \mapsto (u_{i}^{*}(t),\theta_{i}^{*}(t))$ with $T_{m} = 1.5$ minutes. The resulting optimal aggregate power consumption trajectory $P_{\text{total}}^{\text{ref}}(t) = P_{e}\sum_{i=1}^{N}u_{i}^{*}(t)$ is shown as the black curve in Fig. \ref{SetpointTracking}\!\!\!. We emphasize again that the black curve in Fig. \ref{SetpointTracking}\!\!\! is the optimal \emph{planned} aggregate consumption, computed by the LSE ahead of the actual time duration under consideration (in our case, 24 hours ahead). In operation, the LSE also needs to implement \emph{real-time} setpoint control across its customers' TCL population, so as to make the real-time aggregate consumption track the planned optimal aggregate consumption, given the mismatch between the forecasted and real-time ambient temperatures. The brick colored curve in Fig. \ref{SetpointTracking}\!\!\! corresponds to the real-time aggregate consumption for the TCL population with same $T_{m}$, and real-time ambient temperature $\theta_{a}(t)$ as in the solid blue curve in Fig. \ref{PriceAmbient}\!\!\!, for a PID velocity control gain tuple $(k_{p},k_{i},k_{d})$ used to control the setpoint boundaries as part of a mixed centralized-decentralized control. We refer the readers to Section III.1 in Halder \etal. \cite{AbhishekTPS2017} for details on the real-time setpoint control. The purpose of Fig. \ref{SetpointTracking}\!\!\! is to highlight how the solution of the open-loop optimal control problem (\ref{PlanningCostFcn})--(\ref{ComfortRangeConstraint}) can be used by the LSE as a reference aggregate consumption to be tracked in real-time, to elicit demand response.

%%%%%%%%%%%%%%%%%%%%%%%%%%%%%%%%%%%%%%%%%%%%%%%%%%%%%%%%%%%%%%%%%%%%%%%%%%%%%%%%

\section{Concluding Remarks}
In this paper, we have addressed how an aggregator or load serving entity can design an optimal aggregate power consumption trajectory for a population of thermostatically controlled loads. We have formulated this operational planning problem as a deterministic optimal control problem in terms of the day-ahead price forecast, ambient temperature forecast, and an energy budget available from the load forecast. A direct numerical approach to solve the problem is computationally hard. We use tools from optimal control theory to gain analytic insights into the solution of the problem of designing optimal power consumption while respecting individual comfort range constraints. A numerical example is worked out to illustrate how an LSE can use the optimal aggregate power consumption trajectory computed offline, as a reference signal to be tracked in real-time by its customers' TCL population for the purpose of demand response.

%%%%%%%%%%%%%%%%%%%%%%%%%%%%%%%%%%%%%%%%%%%%%%%%%%%%%%%%%%%%%%%%%%%%%%%%%%%%%%%%

%\backmatter

\section*{Acknowledgments}
This work is supported in part by NSF Contract 1760554, ECCS-1546682, NSF Science \& Technology Center Grant CCF-0939370, and the Power Systems Engineering Research Center (PSERC).

%\appendix

\nocite{*}% Show all bib entries - both cited and uncited; comment this line to view only cited bib entries;
\bibliography{wileyNJD-AMA}%

\end{document}